\documentclass[12pt,oneside,reqno]{amsart}
\usepackage{geometry}                
\usepackage{graphicx}
\usepackage{amssymb}
\usepackage{epstopdf}
\usepackage{cases}

\usepackage{amsmath,amsthm,amscd, yhmath}
\usepackage{latexsym, relsize}
\usepackage[colorlinks,citecolor=red,pagebackref,hypertexnames=false]{hyperref}
\usepackage{IEEEtrantools}
\usepackage{appendix}
\usepackage{soul}

\usepackage{hyperref}

\allowdisplaybreaks

\newcommand{\vs}{\vskip.125in}

\newcommand{\R}{\ensuremath{\mathbb{R}}}

\newcommand{\al}{\alpha}

\newcommand{\ga}{\gamma}
\newcommand{\de}{\delta}

\newcommand{\eps}{\epsilon}

\newcommand{\abs}[1]{\left \vert#1\right \vert}
\usepackage[colorlinks,citecolor=red,pagebackref,hypertexnames=false]{hyperref}
\usepackage{mathtools}
\usepackage{float}

\usepackage{xcolor}
\usepackage[matha,
mathx
]{mathabx}

\definecolor{blue}{rgb}{0,0,1}

\definecolor{red}{rgb}{1,0,.2}

\usepackage{ulem} 

\numberwithin{equation}{section}
\theoremstyle{plain}
\newtheorem{theorem}{Theorem}[section]
\newtheorem{lemma}[theorem]{Lemma}
\newtheorem{corollary}[theorem]{Corollary}
\newtheorem{proposition}[theorem]{Proposition}

\theoremstyle{definition}
\newtheorem{definition}[theorem]{Definition}

\theoremstyle{remark}
\newtheorem{remark}[theorem]{Remark}
\newtheorem{note}[theorem]{Note}

\newtheorem{case[theorem]}{Case}

\newcommand*{\myproofname}{Proof of Claim}

\author{Elizabeth G. Campolongo and Krystal Taylor}
\address{Department of Mathematics, The Ohio State University}
\email{campolongo.4@osu.edu  }
\address{Department of Mathematics, The Ohio State University}
\email{taylor.2952@osu.edu}
\thanks{Taylor is supported in part by the Simons Foundation Grant 523555.}
\title{\parbox{14cm}{\centering{Lattice Points close to the Heisenberg Spheres}}}
\begin{document}

\begin{abstract} 
We study a lattice point counting problem for spheres arising from the Heisenberg groups. 
In particular, we prove an upper bound on the number of points
 on and near large dilates of the unit spheres generated by the anisotropic norms $\|(z,t)\|_\al = ( \abs{z}^\al + \abs{t}^{\al/2})^{1/\al}$ for $\al\geq 2$. 
As a first step, we reduce our counting problem to one of bounding an energy integral. 
 The primary new challenges that arise are the presence of vanishing curvature and uneven dilations.  
 In the process, we establish bounds on the Fourier transform of the surface measures arising from these norms.
Further, we utilize the techniques developed here to estimate the number of lattice points in the intersection of two such surfaces.
%
%
\end{abstract}
\maketitle
\setcounter{tocdepth}{1}
\tableofcontents
\section{Introduction}
Estimating the number of integer lattice points that are in, on, and near convex surfaces is a classic subject in number theory, harmonic analysis, and related areas of mathematics.
The Gauss circle problem is a time honored example. 
It concerns approximating the number of integer lattice points in a large dilate of the unit disc in terms relative to the area of the disk: 
\begin{equation}
\#\left( \mathbb{Z}^2 \cap \{x \in \mathbb{R}^2 : \abs{x} \leq R\} \right) = \pi R^2 + E(r),
\end{equation}
where 
$\#(\cdot)$ denotes the size of a finite set. 
It is conjectured that the best bound for the error is $|E(R)| \lesssim R^{1/2 +\epsilon}$.\footnote{Here,  $x \lesssim y$ means that there exists some constant $C$ such that $x \leq Cy$, and 
$x \sim y$ means that both $x \lesssim y$ and $y \lesssim x$.}
 In 2003, Huxley \cite{Huxley3} utilized advanced exponential sum estimates to show that $|E(R)| \lesssim R^{\frac{131}{208}}$. More recently, Bourgain and Watt improved this slightly to $|E(R)| \lesssim R^{\frac{517}{824}}$ in \cite{BW}, building on their work in \cite{BW1}.
\vs

Variants of the Gauss circle problem in which circles are replaced by more intricate surfaces have attracted much attention over the years (see, for instance, \cite{Andrews, Chamizo, IKKN, L10, Schmidt}).  
Notably, Lettington studied the number of integer lattice points near smooth surfaces with non-vanishing curvature (\cite{L10}). 
He established that for $d\geq 3$, $\delta>0$, 
\begin{equation}
\# \left(\{k \in \mathbb{Z}^d \ : R \leq \Vert k \Vert_B \leq R + \delta \} \right) \leq C \max\{R^{d-2+ \frac{2}{d+1}}, R^{d-1}\delta\},
\label{L10}
\end{equation}
where $C>0$ is a universal constant and
$$\Vert x \Vert_B = \inf \{t >0 : x \in tB \}.$$
This result was previously known to hold when $d=2$ (see the discussion in \cite[Introduction]{L10}), and is known to be sharp (see the discussion in \cite{IT}). 
A simple Fourier analytic proof of Lettington's result, as well as an extension to a variable coefficient setting, was provided by Iosevich and the second listed author in \cite{IT}.
\vs

The work in \cite{IT}, as well as Lettington's original result, are contributions towards a conjecture of W. Schmidt \cite{Schmidt}. 
This conjecture concerns the number of lattice points on a given surface with non-vanishing curvature, and states that if B $\subset \mathbb{R}^d$, for $d \geq 3$, is a symmetric convex body with a smooth boundary that has non-vanishing Gaussian curvature, then for any $\epsilon >0$,
$$\# \left(\{R\partial B \cap \mathbb{Z}^d \} \right) \leq C_{ \epsilon} R^{d-2+ \epsilon }.$$
\vs

A more complex problem arises when we consider lattice point counting problems for surfaces with points of vanishing curvature. 
This includes error estimates in the case of super spheres \cite{Krtl2, Krtl3} and $\ell^p$-balls \cite{Randol1, Randol2}, as well as more general surfaces of rotation \cite{KN1, Nowak}.  
See also \cite{Peter1, Peter2} where the effect of points of vanishing curvature on the error term or ``the lattice remainder term'' is studied in much greater generality.  
\vs
{
An example more specific to this paper arises from  the
family of
 Heisenberg norm balls, defined by 
\begin{equation}\label{Hball}
B_R^{\alpha,A} := \big\{(z,t) \in \mathbb{R}^{2d} \textsf{ x } \mathbb{R}: |z|^\alpha + A|t|^{\alpha/2} \leq R^\alpha\big\}.
\end{equation}
Depending on which value of $\alpha$ is considered, the surface of $B_R^{\alpha,A} $
has points with curvature vanishing to maximal order; more details are presented in Section \ref{Curvature&Decay}.   
The Heisenberg norm balls were considered by 
Garg, Nevo, and the second listed author in 2014 \cite{GNT}.  They investigated a variant of the Gauss circle problem replacing Euclidean balls by those in \eqref{Hball} and provided bounds on the error term, which they demonstrated were sharp when $\alpha=2$ in all dimensions. 
The sharpness when $\alpha=4$ was subsequently demonstrated by Gath \cite{Gath, Gath2}.

\subsection{Main Result}
In this paper, we provide an upper bound on the number of lattice points 
near the surfaces of the Heisenberg norm balls. 
These surfaces present an interesting evolution in lattice point counting problems due to their points of vanishing curvature. 
Before stating our main result, we give a brief overview of the Heisenberg group: we refer to  \cite[Chapter 1]{Folland2} or \cite[Chapter 12]{Stein} for a deeper treatment.
  \vs

The \textbf{Heisenberg group} is defined to be 
$$\mathsf{H}_d = \ \mathbb{R}^{2d}  \ \mathsf{x} \ \mathbb{R} =  \{(z,t) :\in \mathbb{R}^{2d}, t \in \mathbb{R}\},$$ 
where
 multiplication is defined 
 by
$$(x,y,t)\cdot(u,v,s) = \big(x+u, y+ v, t+s - \frac{1}{2}\big(\langle x,v \rangle - \langle u,y \rangle\big)\big),$$ 
for $x,y,u,v \in \R^d$ and $s,t\in \R$, 
where $\langle x,v \rangle$ is the standard inner product on $\mathbb{R}^d$ for $d\geq 1$. 
\vs

The \textbf{Heisenberg norm} is defined by the following family of gauge functions on the Heisenberg group
\begin{equation}\label{HeisenbergNorm}
\Vert(z,t)\Vert_{\alpha,A} = \big(\abs{z}^{\alpha}+A\abs{t}^{\alpha/2}\big)^{1/\alpha},
\end{equation}
for $\alpha, A > 0$, and where $z \in \mathbb{R}^{2d}, t \in \mathbb{R}$ and $|\cdot|$ denotes the Euclidean norm. 
When $\alpha=4$, this anisotropic norm is more commonly referred to as the \textit{Cygan} \cite{Cy1, Cy} and \textit{Kor\'{a}nyi} \cite{Kor} norm. For simplicity, we will assume $A=1$.
\vs

A useful manipulation of elements of the Heisenberg group is the \textbf{Heisenberg dilation}, which is a class of functions indexed by $a \in \mathbb{R}_+$ and defined as $\varphi_a : (z,t) \mapsto (az,a^2t)$. These dilations describe an automorphism group of $\mathsf{H}_d$ through which a notion of homogeneity on the group is understood.
\vs

We note that the Heisenberg groups are examples of non-abelian Lie groups, and as such play an active role in modern harmonic analysis. A number of recent results have appeared investigating mapping properties of maximal averaging operators associated with the Heisenberg groups.  See, for instance, \cite{JSS21} and the references therein. 
\vs

With these concepts in tow, we define the \textbf{Heisenberg sphere} (or Kor\'{a}nyi} sphere) of radius $R$ and index $\alpha$ 
by 
\begin{equation}\label{surfaces}
\partial B_R^\alpha = \left\{   (z,t) \in \R^{2d}\times \R:  \Vert (z,t) \Vert_{\alpha} = R \right\}.
\end{equation}
A particularly interesting feature of these surfaces of revolution is that they have points of maximal vanishing curvature when $\al>2$; this is discussed further in Section~\ref{Curvature&Decay}. \vs

Our main result on the number of lattice points on and near the Heisenberg spheres follows.

\begin{theorem}\label{Heisenberg}
Let $d\geq 1$ and $\al\geq 2 $ be integers, and set $n=2d+1$. For $(z,t) \in \R^{2d}\times \R$, let $\|(z,t) \|_\al$ be the Heisenberg norm defined by 
$$\|(z,t) \|_\al= \big( \abs{z}^\alpha + \abs{t}^{\frac{\al}{2}} \big)^{\frac{1}{\al}}.$$
Then
\begin{equation}\label{Heisenberg_count}
 \#\big( \{ m \in \mathbb{Z}^{n} :  
R-\delta \leq \Vert m \Vert_\alpha \leq R+ \delta \}\big) 
 \lesssim 
 \begin{cases}
 \max\{R^{\,n- \frac{1}{(\al-1)} }, R^n\delta\},             \text{ if }  \al \geq 2(n-1),\\
 \max\{ R^{n- \frac{1}{ (2n-3) } },       R^n\delta\},  \!  \text{ if } 2< \al < 2(n-1),\\
\max\{ R^{n-  \frac{(n-2)}{n}  },       R^n\delta\},     \,  \text{ if } \al=2, \\
 \end{cases}
 \end{equation}
for $0 \leq \delta<1$ and $R>1$.
\end{theorem}
\vs

In Section \ref{sharp_section}, we will see that Theorem \ref{Heisenberg} is sharp, meaning that \eqref{Heisenberg_count} is an equality, for all $\al$ provided $\de$ obeys a lower bound dependent on $\al$ and $n$.  
\vs

We now proceed by giving some further context to Theorem \ref{Heisenberg} with some commentary and comparisons to known literature.  We begin by observing a simple trivial bound. 
\vs 

\subsubsection{Comparison of Theorem \ref{Heisenberg} to a trivial bound}
For large values of $R$, 
the left-hand-side of \eqref{Heisenberg_count} is bounded trivially by 
$$\left|B_{(R+2)}^\alpha\right| - \left|B_{(R-2)}^\alpha \right|,$$
where 
$|\cdot |$ is used to denote the Euclidean volume and $B_R^\alpha$ is 
a dilate by $R$ of the unit ball under the map $\varphi_R : (z,t) \rightarrow (Rz,R^2t)$.
As such, $|B_R^\alpha|=R^{2d+2}|B_1^\alpha|$, which is evident upon noting that
\begin{IEEEeqnarray}{cl}
|B_R^\alpha| &= \int_{ \left\{ (z,t) \in \mathbb{R}^{2d} \ \mathsf{x} \ \mathbb{R} \ : \ |z|^\alpha + |t|^{\alpha/2} \leq R^\alpha \right\} } dzdt 
\nonumber \\
&= \int_{ \left\{ (z,t) \in \mathbb{R}^{2d} \ \mathsf{x} \ \mathbb{R} \ : \left|\frac{z}{R}\right|^\alpha + \left|\frac{t}{R^2} \right|^{\alpha/2} \leq 1 \right\} } dzdt .
\nonumber
\end{IEEEeqnarray}
Making the change of variables, $w = \frac{z}{R} \in \mathbb{R}^{2d}$ and $s = \frac{t}{R^2} \in \mathbb{R}$, yields
\begin{IEEEeqnarray}{cl}
|B_R^\alpha| &= R^{2d+2} \int_{ \left\{ (w,s) \in \mathbb{R}^{2d} \ \mathsf{x} \ \mathbb{R} \ : \ |w|^\alpha + |s|^{\alpha/2} =1 \right\} } dzdt 
\nonumber \\
&= R^{2d+2}|B_1^\alpha|.
\nonumber
\end{IEEEeqnarray}
In conclusion, the expression on the left-hand-side of \eqref{Heisenberg_count} is bounded trivially by a constant multiple of $R^{2d+1} = R^n$, and our bound is an improvement for all $\al$.  
\vs

\subsubsection{Comparison of Theorem \ref{Heisenberg} to Iosevich-Taylor's \cite[Theorem 1.1]{IT}}
The first step in the proof of Theorem \ref{Heisenberg} consists of transforming our lattice point counting problem into one of bounding an energy integral of a fractal-like set. This idea was 
 introduced by Iosevich and Taylor in \cite{IT}.  
 However, 
due to the irregular scaling inherent to the geometry of the Heisenberg norms, the arising energy integral in our context is more complex than those handled in \cite{IT}.  
Consequently, a much more delicate analysis is required, as well as
decay estimates on the Fourier transforms of the surface measures associated with the Heisenberg norms. 
\vs

Though it is possible to expand the method in \cite{IT} to surfaces with vanishing curvature--as the first listed author does in her thesis \cite[Proposition 1.3.5]{Campo}--the Heisenberg spheres fall outside its scope.
In more detail, Campolongo's extension of \cite[Theorem 1.3]{IT} does not apply  to the Heisenberg spheres to prove Theorem \ref{Heisenberg}. 
Indeed, in the language of \cite[Proposition 1.3.5]{Campo} or \cite[Theorem 1.3]{IT}, 
set $\al_i = a$ for each $i=1,\dots, (n-1)$ and $\al_n= 2a$, where $a= \frac{n}{(n+1)}$. 
Observing that in order to apply this proposition we must have
$2a \le \frac{n}{n-\gamma}$, 
where $\gamma$ is the decay exponent on the Fourier side for the surface,
we see it is necessary that $n-\ga \le \frac{n}{2a}$, which is equivalent to $\frac{n-1}{2} \le \ga$. 
Since the Heisenberg spheres
have points of vanishing curvature, such a lower bound on $\ga$ is not possible. In particular, we show in Proposition~\ref{alphaDecay} that for any given $\alpha \geq 3$, the decay is $\ga = \min\!\left\{\frac{n-1}{\al}, \frac{1}{2}\right\} <\frac{n-1}{2}$ for $n \geq 3$, and for $\alpha =2$ that $\gamma = \frac{n-2}{2} < \frac{n-1}{2}$, thus excluding such an application. 
\vs
%
%
%
\subsubsection{Comparison of Theorem \ref{Heisenberg} to Garg-Nevo-Taylor's \cite{GNT}}
As a final comment on Theorem \ref{Heisenberg}, 
in Section \ref{sec_alt} we consider an alternate way to bound the lattice point count based on the work of \cite{GNT}.  This alternate method provides an improved bound for small values of $\delta>0$, and this result is stated in  Proposition~\ref{useGNT}. Our proof of Theorem \ref{Heisenberg}, 
on the other hand, offers a more direct approach in that it does not rely on the work in \cite{GNT} and all oscillatory integrals are computed directly without the use of Bessel functions. 
\vs

\subsection{Intersections of Heisenberg norm balls} 
We finish this section by considering a lattice point count near the intersection of two Heisenberg spheres  \eqref{surfaces}.    
For $0<\Vert p \Vert_\al\leq CR$ fixed, we are interested in bounding 
\begin{equation}\label{dream_HNBIntersect}
\#\big(\big\{ m \in \mathbb{Z}^{n} :   |\Vert m \Vert_\al - R| \leq \delta \text{ and } |\Vert m-p \Vert_\al - R| \leq \delta \big\}\big).
\end{equation}

For a typical $p$, this will be the number of lattice points near the intersection of two Koryani spheres, and this intersection is maximized when $p$ is near the origin.  The following average bound follows as an immediate consequence of Theorem \ref{Heisenberg}.  

\begin{corollary}[Intersecting Heisenberg Spheres]\label{HNBInt}
Let $d \geq 1$ and $\alpha \geq 2$ be integers, set $n = 2d+1$, and 
let $\|\cdot \|_\al$ be the Heisenberg norm defined by 
$$\|(z,t) \|_\al= \big( \abs{z}^\alpha + \abs{t}^{\frac{\al}{2}} \big)^{\frac{1}{\al}}.$$
Then for $R>1$
and $0 \leq \delta < 1$,
\begin{align*} 
\#\big(\{(m,p) \in \mathbb{Z}^{n} \, \mathsf{x} \ \mathbb{Z}^{n} \!:  
|\Vert m &\Vert_\al \!- R| \!\leq \delta \text{ and } |\Vert m-p \Vert_\al \!- R| \!\leq \delta \}\big) \nonumber\\
&\lesssim
\max\{ R^{2 \left(n - \frac{ \gamma(\al) }{n-1 - \gamma(\al)   }\right)},       \left(R^n\delta\right)^2 \},
\end{align*}
where  $\ga(\al) = \min\!\left\{\frac{n-1}{\al}, \frac{1}{2}\right\}$ when $\al \geq 3$, 
and 
$\ga(\al) = \frac{n-2}{2}$ when $ \al =2$. 
\end{corollary}

The proof is very short so we include it here.  
\begin{proof}[Proof of Corollary~\ref{HNBInt}]
Observe that 
\begin{equation} \label{IntSet}
\#\big(\big\{(m,p) \in \mathbb{Z}^{n} \ \mathsf{x} \ \mathbb{Z}^{n} \!:  |\Vert m \Vert_\al - R| \!\leq \delta \text{ and } |\Vert m-p \Vert_\al - R| \!\leq \delta \big\}\big)
\end{equation}
is simply 
\begin{equation} \label{setProd}
\big(\#\big\{m \in \mathbb{Z}^{n} \!:\!  |\Vert m \Vert_\al  - R| \!\leq \delta \big\} \big) \!
\cdot \!\big(\#\big\{p \in \mathbb{Z}^{n} \! : \! |\Vert p \Vert_\al - R| \!\leq \delta \big\}\big)\!.
\end{equation}
From Theorem~\ref{Heisenberg}, we know that the size of each of these sets is simply bounded by 
 \begin{equation}
 \#\big( \{ m \in \mathbb{Z}^{n} :  
R-\delta \leq \Vert m \Vert_\alpha \leq R+ \delta \}\big) 
 \lesssim \max\{ R^{n - \frac{ \gamma(\al) }{n-1 - \gamma(\al)   }},       R^n\delta\}.
 \end{equation} 
Thus, \eqref{HNBInt} $\lesssim \max\{ R^{2 \left(n - \frac{ \gamma(\al) }{n-1 - \gamma(\al)   }\right)},       \left(R^n\delta\right)^2 \}$.
\end{proof}
\vs

\begin{remark}
We note that while the proof of Corollary \ref{HNBInt} is very short, it sets the stage for a deeper investigation of lattice points near the intersection of two more general surfaces.  In light of recent developments on $L^2$ bounds on operators of the form $\sigma*\mu$ (see \cite{BIT, ITtrees}), it would be interesting to study intersections of two surfaces in the non-isotropic setting investigated in \cite{IT}.  
\end{remark} 
\vs

\subsection{Structure}\label{structure_sec}
Theorem~\ref{Heisenberg} is proved in Sections \ref{Method}-\ref{GeoSum}. 
We first reduce the proof to two main components: (1) a decay estimate and (2) an energy estimate. These preliminary reductions are done in Section~\ref{Method}. The decay estimates, along with the necessary curvature computations, appear in Section~\ref{Curvature&Decay}.
Section~\ref{GeoSum} is dedicated to bounding the energy integral. 
In Section \ref{sharp_section}, we demonstrate that Theorem \ref{Heisenberg} is sharp provided $\de$ is bounded below and give an improvement for smaller $\de $.

\subsection*{Acknowledgments}
We would like to acknowledge Professor Allan Greenleaf at the University of Rochester for sharing invaluable insights and feedback that greatly improved our article.  Thank you for your endless patience and kindness.  
\vskip.125in

\section{Method: From Lattice Points to Fractal Geometry}\label{Method}
In this section, we use a series of lemmas to break down the proof of Theorem \ref{Heisenberg} and reduce matters to two key propositions.  Throughout, $n=2d+1$, where $d\geq 1$ is an integer. 
\vs

As a technical point, in order to prove Theorem \ref{Heisenberg}, we will first establish that 
\begin{equation}\label{star}
q^{-n}\#( \{ (n,m) \!\in \mathbb{Z}^n \ \mathsf{x} \ \mathbb{Z}^n : \Vert n \Vert_\alpha, \Vert m \Vert_\alpha \leq Cq^a, \, \abs{ \Vert n-m \Vert_\alpha - q^a} \leq q^{a-\tau} \}) \lesssim q^{n-\tau},
\end{equation}
for 
$\tau \in \left(a,    \frac{ (n-1)a}{ n-1-\ga(\al)}   \right]$, where $\ga(\al) $ is defined as in Theorem \ref{Heisenberg} (and derived in Proposition \ref{alphaDecay}) below). 
Setting $R=q^a$ and $\de = q^{a-\tau}$, then plugging in values of $\tau$ yields the result; further details of this deduction are given in Section \ref{reduction_section}. 
Regarding the bounds on $\tau$, the lower bound $a< \tau$ results in the natural restriction that $\de< 1$, and the upper bound on $\tau$ is a requirement of our proof technique and will make an appearance in Section \ref{GeoSum}.

\subsection{An overview of the proof of \eqref{star}}

We first truncate and scale the lattice according to the geometry of the surfaces of the Heisenberg spheres defined in \eqref{surfaces}. 
Next, we define a measure $\mu_q$ on the resulting scaled and thickened lattice (see Definition \ref{defMuQalpha}).  
Through a series of lemmas, we reduce the proof of \eqref{star} to that of establishing Proposition \ref{alphaDecay} (a decay estimate) and Proposition \ref{energy} (an energy bound). 

\subsection{Reduction of the proof } 
\label{reductSec}
We rely on the following construction, which is similar to that used in \cite{IT}.  
Recall $n=2d+1$ and consider the truncated set of lattice points:
\begin{equation}
 L_{n,q}:= \big\{(b_1, \dots, b_n) \! : \! b_i \in \{0,1, \dots, \lceil q^a \rceil \}, i \neq n, \!\text{ and } b_n \in \{0,1, \dots, \lceil q^{2a} \rceil \} \!\big\},
\end{equation}
where $a$ is chosen so that the exponents sum to the dimension: $a(n-1) + 2a = n$; namely,  
$$a := \frac{n}{n+1}.$$

As thickening the surface to perform the lattice point count on the left-hand-side of \eqref{star} is in some sense equivalent to first ``thickening'' the lattice, we 
consider the $q^{a-\tau}$-neighborhood of $L_{n,q}$ for $\tau$ in a range to be determined.  
As it will lend better to the geometry, we replace each $(q^{a-\tau})$-\textit{ball} by a $(q^{a-\tau})$-\textit{box} with the same center.
\vs

Scaling $L_{n,q}$ down into the unit box, we set
\begin{equation}\label{EqDef}
E_q:= \bigcup_{(b_1, \dots, b_n) \in L_{n,q}} \left\{ R_\tau + \left(\frac{b_1}{q^a}, \dots, \frac{b_{2d}}{q^a}, \frac{b_n}{q^{2a}}\right)\right\},
\end{equation}
where $R_\tau$ denotes the $q^{-\tau} \ \mathsf{x} \cdots \mathsf{x} \ q^{-\tau} \ \mathsf{x} \ q^{-a-\tau}$-rectangular box centered at the origin in $\mathbb{R}^{n}$. 
Note $E_q$ is composed of approximately $q^n$ such rectangles.  
\begin{figure}[ht]
  \centering
  \includegraphics[width=\linewidth]{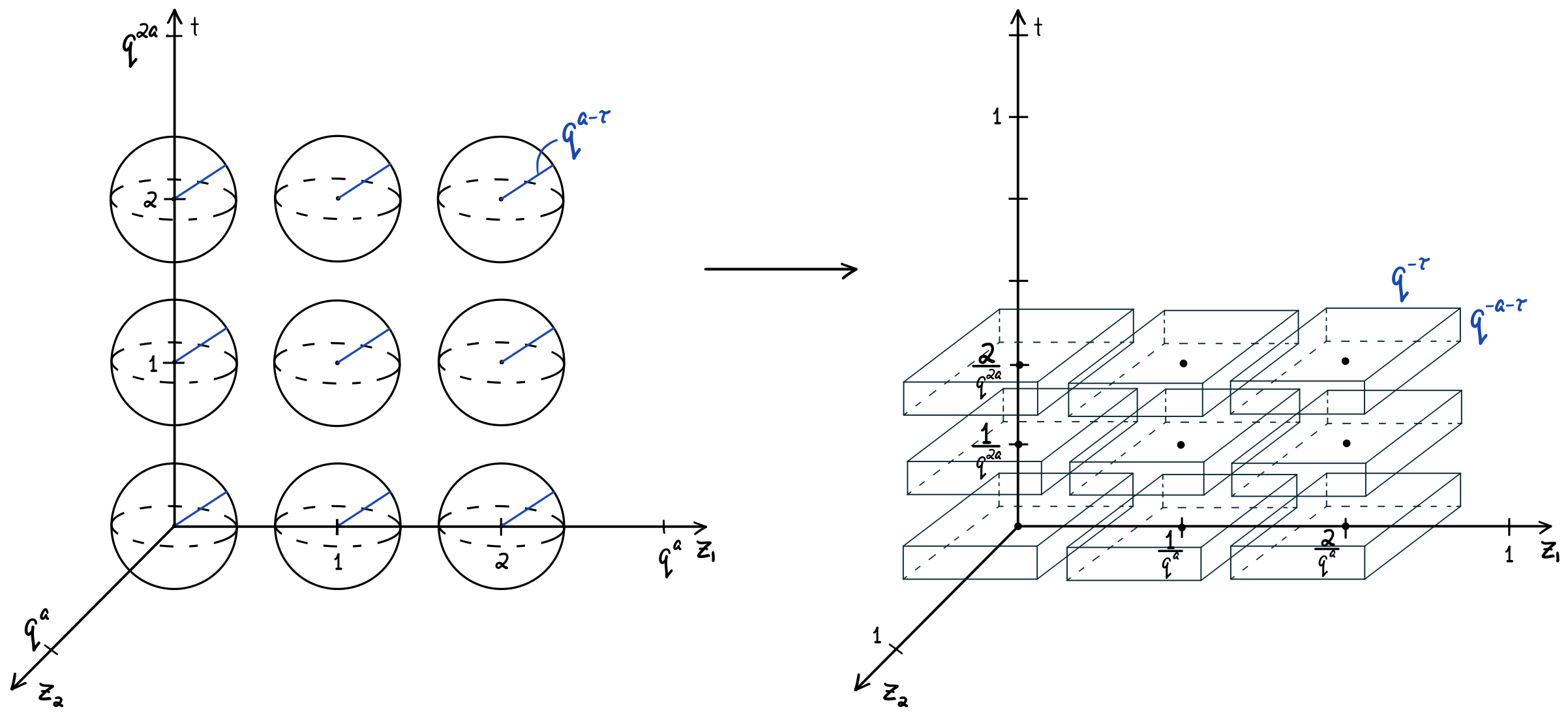}
\caption{Example of $L_{3,q}$ and the scaling to $E_q$.}\label{latticeScale}
\end{figure}
To avoid overlap amongst the rectangles, we require $\tau$ to satisfy 
\begin{equation}\label{first_tau_restriction}
a:= \frac{n}{n+1}< \tau.
\end{equation}
Denote the volume of $R_\tau$ by $V_R$ and observe that
\begin{equation} \label{volRect}
V_R := vol(R_\tau) = (q^{-\tau})^{n-1}q^{a-\tau}=q^{-n\tau-a}.
\end{equation}

We now define a probability measure, $\mu_q$, on $E_q$.  While the definition of $\mu_q$ is quite technical looking, it is simply a normalized and smoothed version of the indicator function of the set $E_q$.
\begin{definition}[Probability measure on $E_q$]\label{defMuQalpha}
For $x \in \mathbb{Z}^{2d+1}$ and $\tau$ satisfying \eqref{first_tau_restriction}, we define $\mu_q$ 
by:
\begin{equation*}
\mu_q(x) \! := 
\frac{1}{\abs{E_q}} \sum_{b \in \mathbb{Z}^n} \left[\prod_{i=1}^{2d} \psi_0\!\left(\frac{b_i}{q^a}\right)  
\psi_0\!\left(\!q^{\tau}\! \left(x_i-\frac{b_i}{q^a}\right)\!\right)\right] \! \psi_0\!\left(\frac{b_n}{q^{2a}}\right) \psi_0\!\left(\!q^{a+\tau}\! \left(x_n-\frac{b_n}{q^{2a}}\right)\!\right) \!,
\end{equation*}
where $\psi_0$ is a bump function supported on the unit ball. 
For $A\subset E_q$, set 
$\int_A d\mu_q(x) = \int_A \mu_q(x) dx.$ 
\end{definition}

The following lemma transforms the
lattice point counting problem to one of bounding the $
\mu_q\times \mu_q$-measure of 
pairs of points 
spaced approximately a distance $1$ apart.

\begin{lemma}[Approximating points by measure]\label{DecayEquivHNB}
\begin{align}
q^{-n}\#\big(\big\{ (n,&m) \in \mathbb{Z}^{n} \ \mathsf{x} \ \mathbb{Z}^{n} : \Vert n \Vert_\alpha, \Vert m \Vert_\alpha. \leq Cq^a, \abs{ \Vert n-m \Vert_\alpha - q^a} \leq q^{a-\tau} \big\}\big) \nonumber\\
&\sim q^n \mu_q  \ \mathsf{x} \ \mu_q \left( \big\{ (x,y) \in E_q \ \mathsf{x} \ E_q : \abs{\Vert x-y \Vert_{\alpha} - 1} \leq q^{-\tau} \big\} \right) \label{measFS}
\end{align}
\end{lemma}

\begin{proof}
$E_q$ is the support of $\mu_q$, as described above. Let $N_C(S, \rho)$ be the number of cubes of side-length $\rho$ and $N_R(S,\rho_1, \rho_2)$ be the number of $\rho_1  \ \mathsf{x} \cdots \mathsf{x} \ \rho_1  \ \mathsf{x}  \ \rho_2$ rectangular boxes required to cover a set $S$. Now, defining $\tau_q^{\alpha}(x) =(q^ax_1, \dots, q^ax_{2d}, q^{2a}x_{n})$, we claim:
\begin{align}
\label{MeasFSsim0}
(&q^{-n})^2 \# \left(\big\{(n,m) \in \mathbb{Z}^n \mathsf{x} \ \mathbb{Z}^n : \Vert n \Vert_{\alpha}, \Vert m \Vert_{\alpha} \!\leq Cq^a, \abs{\Vert n-m \Vert_\alpha - q^a} \leq q^{a-\tau} \big\}\right)
\\
&\sim (q^{-n})^2N_C \left(\big\{(u, v) \in \tau_q^\alpha(E_q) \ \mathsf{x} \ \tau_q^\alpha(E_q) :  \abs{\Vert u-v \Vert_\alpha - q^a} \leq q^{a-\tau} \}, q^{a-\tau}\right) 
\label{MeasFSsim1}\\
&\sim (q^{-n})^2N_R \left(\big\{(x,y) \in E_q \ \mathsf{x} \ E_q :  \abs{\Vert x-y \Vert_\alpha - 1} \leq q^{-\tau} \}, q^{-\tau}, q^{-a-\tau}\right)
\label{MeasFSsim2}\\
&\sim \mu_q \ \mathsf{x} \ \mu_q \left(\big\{(x,y) \in E_q \ \mathsf{x} \ E_q : \abs{\Vert x-y \Vert_\alpha - 1} \leq q^{-\tau} \big\}\right)
\label{MeasFSsim3}
\end{align}
\end{proof}

Next, using elementary Fourier analysis, we have the following equivalence. 

\begin{lemma}  \label{FALem}
Let $\sigma_{\alpha}$ denote the surface measure on the Heisenberg sphere defined in \eqref{surfaces} with $R=1$ (a formal discussion of surface measures is given below in Section \ref{SurfMeasDecay}). It follows that
\begin{equation}\label{Fourier}
 \mu_q  \ \mathsf{x} \ \mu_q \left( \big\{ (x,y) \in E_q \ \mathsf{x} \ E_q : \abs{\Vert x-y \Vert_{\alpha} - 1} \leq q^{-\tau} \big\} \right)
\! \sim \! q^{-\tau} \int \abs{\widehat{\mu}_q(\xi)}^2 \widehat{\sigma}_\alpha(\xi) d\xi. 
\end{equation}
\end{lemma}
\begin{proof}
Let $\epsilon = q^{-\tau}$. Then
\begin{align}
 \mu_q \ \mathsf{x} \ \mu_q & \left( \{ (x,y) \in \mathbb{Z}^{n} \ \mathsf{x} \ \mathbb{Z}^{n}: 1 - \epsilon \leq  \Vert x-y \Vert_\alpha  \leq 1+ \epsilon \} \right) \\
&=\iint_{ \{x : \abs{\Vert x-y \Vert_\alpha - 1} < \epsilon\}} d\mu_q(x)d\mu_q(y) \\
&=  \iint \chi_{A_{1,\epsilon}}(x-y)d\mu_q(x)d\mu_q(y), \label{chiAInt}
\end{align}
where $A_{1,\epsilon}:= \{z : \abs{\Vert z \Vert_\alpha - 1} < \epsilon\}$. \\

A computation shows that 
$\frac{1}{\epsilon} \chi_{A_{1,\epsilon}}(\cdot) \sim \sigma_{\alpha} \!*\! \rho_{c\epsilon}(\cdot)$, 
where $\rho$ is a non-negative smooth bump function satisfying 
$spt(\rho) \subseteq B_2(0)$ and $\rho \equiv 1$ on $B_1(0)$, 
and $\rho_\epsilon(\cdot) = \left(\frac{1}{\epsilon}\right)^n \rho \!\left(\frac{\cdot}{\epsilon}\right)$.
Now, through this relation, we may rewrite \eqref{chiAInt} as approximately
\begin{equation*}
\epsilon\iint \sigma_\alpha \!*\! \rho_\epsilon(x-y)d\mu_q(x)d\mu_q(y).
\end{equation*}
Using Fourier inversion and Fubini, we have
\begin{align}
\eps \iint \sigma_\alpha \!*\! \rho_\epsilon(x-y)d\mu_q(x)d\mu_q(y)
&= \eps \iiint \widehat{\sigma_\alpha \!*\! \rho_\epsilon}(\xi)e^{2\pi i \xi(x-y)} d\xi d\mu_q(x)d\mu_q(y)\\
&= \eps \int  \widehat{\sigma}_\alpha(\xi) \widehat{\rho}(\epsilon\xi) \widehat{\mu}_q(\xi)   \overline{\widehat{\mu}_q(\xi)} d\xi.\\
&= \eps \int \abs{\widehat{\mu}_q(\xi)}^2 \widehat{\sigma}_\alpha(\xi) \widehat{\rho}(\epsilon\xi)  d\xi.  \label{FAStep}
\end{align}

Finally, plugging $\epsilon = q^{-\tau}$ and observing that $\Vert \widehat{\rho} \Vert \sim 1$, 
\eqref{Fourier} follows.
\end{proof}

Next, we require an estimate on $\widehat{\sigma}_\al$.  
Due to the presence of vanishing curvature, known oscillatory integral estimates are not sufficient for our purposes, 
and we will prove the following proposition. 
Curvature and its role in oscillatory integral estimates is the topic of Section \ref{Curvature&Decay}. 
\vs

\begin{proposition}[Decay estimate for Heisenberg surface measure]\label{alphaDecay}
With 
$\sigma_{\alpha}$ as in Lemma \ref{FALem}, 
$$\abs{\widehat{\sigma}_\alpha(\xi)} \lesssim (1+ \abs{\xi})^{-\ga(\al) },$$
where  $\ga(\al) = \min\!\left\{\frac{n-1}{\al}, \frac{1}{2}\right\}$ when $\al \geq 3$, 
and 
$\ga(\al) = \frac{n-2}{2}$ when $ \al =2$. 
\end{proposition}
\vs

The proof of Proposition~\ref{alphaDecay} is found in Section~\ref{Curvature&Decay} along with a discussion on the points of vanishing curvature.  
\vs

Next, plugging the estimate of Proposition~\ref{alphaDecay} into \eqref{Fourier}, we see that
\begin{align*}
q^n \mu_q  \, \mathsf{x} \, \mu_q \big( \big\{ (x,y) \!\in \! E_q \ \mathsf{x} \ E_q \!:\! \abs{\Vert x-y \Vert_{\alpha} \!- \!1} \!\leq q^{-\tau} \big\} \big)
&\sim q^{n-\tau} \int \abs{\widehat{\mu}_q(\xi)}^2 \widehat{\sigma}_\alpha(\xi) d\xi \\
&\lesssim  q^{n-\tau}  \int \abs{\widehat{\mu}_q(\xi)}^2 (1+ \abs{\xi})^{-\gamma(\alpha)} d\xi,
\end{align*}
and using elementary properties of the Fourier transform, 
\begin{equation}\label{energy_2}
 \int_{|\xi|>1} \big|\widehat{\mu_q^\alpha}(\xi)\big|^2 \abs{\xi}^{-\gamma(\alpha)} d\xi 
\sim \! \iint \abs{x-y}^{\gamma(\alpha)-n} d\mu_q^\alpha(x)d\mu_q^\alpha(y),
\end{equation}
which is the $\big(n - \gamma(\al)\big)$-energy of the function $\mu_q$ (see \cite{Falc86} for a discussion of energy and potentials).  
\vs

In conclusion, following the string of relations from Lemma \ref{DecayEquivHNB} to \eqref{energy_2}, we have transformed our counting problem in \eqref{star} to one of bounding an energy integral.  
The following proposition gives a bound for the expression in \eqref{energy_2} provided $\tau$, the dilation factor in the definition of $\mu_q$ of Definition \eqref{defMuQalpha}, is not too large. 

\begin{proposition}[Energy integral bound]\label{energy}
Let $\mu_q$ as given in Definition~\ref{defMuQalpha} and recall $a= \frac{n}{n+1}$.
If $\al\geq 4d$ and 
$
\tau \in \left(a, \frac{a\alpha}{\alpha-1} \right],
$
 then 
$$
\iint \abs{x-y}^{\frac{n-1}{\alpha}-n}d\mu_q(x)d\mu_q(y) \lesssim 1. 
$$
\vs
\noindent
More generally, if $\al\geq 2$ and 
$
\tau \in \left(a,    \frac{ (n-1)a}{ n-1-\ga(\al)}   \right],
$
 then 
$$
\iint \abs{x-y}^{\ga(\al)-n}d\mu_q(x)d\mu_q(y) \lesssim 1,
$$
where $\ga(\al) $ is defined as in Proposition \ref{alphaDecay}. 
\end{proposition}

Proposition \ref{energy} is proved in Section~\ref{GeoSum}. 
\vs

In summary, 
the proof of \eqref{star} is thus reduced to establishing Proposition \ref{alphaDecay} 
and Proposition \ref{energy}, the proofs of which are given in Section \ref{Curvature&Decay} and Section \ref{GeoSum} respectively.  We end this section with a technical computation showing that Theorem \ref{Heisenberg} follows immediately from \eqref{star}. 

\subsection{Deducing Theorem \ref{Heisenberg} from \eqref{star}}\label{reduction_section}
To obtain the inequality in Theorem \ref{Heisenberg} from that in \eqref{star}, we consider two ranges for $\de \in [0,1)$. 
Set $\tau_0 = \frac{ (n-1)a}{ n-1-\ga(\al)} $ for $\al\geq 2$, where $\ga(\al)$ is defined in Proposition \ref{alphaDecay}. 
\vs

If $\de <  q^{ a - \tau_0 }$, then applying \eqref{star} we deduce 
\begin{align*}
q^{-n}&\#( \{ (n,m) \in \mathbb{Z}^n \ \mathsf{x} \  \mathbb{Z}^n : \Vert n \Vert_\alpha, \Vert m \Vert_\alpha
 \leq Cq^a, \ \abs{ \Vert n-m \Vert_\alpha - q^a} \leq \de \}) \\
& \le 
 q^{-n}\#( \{ (n,m) \in \mathbb{Z}^n \ \mathsf{x} \  \mathbb{Z}^n : \Vert n \Vert_\alpha, \Vert m \Vert_\alpha
 \leq Cq^a, \ \abs{ \Vert n-m \Vert_\alpha - q^a} \leq   q^{ a - \tau_0}  \}) \\
&\lesssim 
 q^{ n - \tau_0 }
 = \max\{q^{n-\tau_0}, q^{n-a}\delta\}.
 \end{align*}
 \vs

If $\de \in \big[q^{a - \tau_0}, 1\big)$, write $\de = q^{a-\tau}$ for some 
$\tau \in \left(a,   \tau_0\right]$.  Now, 
\begin{align*}
q^{-n}&\#( \{ (n,m) \in \mathbb{Z}^n \ \mathsf{x} \  \mathbb{Z}^n : \Vert n \Vert_\alpha, \Vert m \Vert_\alpha
 \leq Cq^a, \ \abs{ \Vert n-m \Vert_\alpha - q^a} \leq \de \}) 
\\
& = 
 q^{-n}\#( \{ (n,m) \in \mathbb{Z}^n \ \mathsf{x} \  \mathbb{Z}^n : \Vert n \Vert_\alpha, \Vert m \Vert_\alpha
 \leq Cq^a, \ \abs{ \Vert n-m \Vert_\alpha - q^a} \leq   q^{ a - \tau}   \}) \\
 &\lesssim 
 q^{ n - \tau}
  = \max\{q^{n-\tau_0}, q^{n-\tau}\}
 = \max\{q^{n- \tau_0}, q^{n-a}\delta\}.
\end{align*}
From here, we observe that
\begin{align}
\#( \{ (n,&m) \in \mathbb{Z}^n \ \mathsf{x} \  \mathbb{Z}^n : \Vert n \Vert_\alpha, \Vert m \Vert_\alpha
 \leq Cq^a, \ \abs{ \Vert n-m \Vert_\alpha - q^a} \leq \de \}) \\
 &\sim q^n\#( \{ (0,m) \in \mathbb{Z}^n \ \mathsf{x} \  \mathbb{Z}^n : \Vert m \Vert_\alpha
 \leq Cq^a, \ \abs{ \Vert 0-m \Vert_\alpha - q^a} \leq \de \}) \\
 &= q^n\#\big( \{ m \in \mathbb{Z}^n :
q^a -\delta \leq \Vert m \Vert_\alpha \leq q^a + \delta \}\big), 
 \end{align}
since the Heisenberg norm is translation invariant.
Upon replacing $q$ with $R^{1/a}$, where $a=\frac{n}{n+1}$, we conclude that 
 \begin{equation}
 \#\big( \{ m \in \mathbb{Z}^{n} :  
R-\delta \leq \Vert m \Vert_\alpha \leq R+ \delta \}\big) 
 \lesssim \max\{ R^{n - \frac{ \gamma(\al) }{n-1 - \gamma(\al)   }},       R^n\delta\},
 \end{equation} 
and
Theorem \ref{Heisenberg} follows immediately upon 
plugging in the values of $\ga(\al)$ obtained in Proposition \ref{alphaDecay}.
\vskip.125in

\section{Curvature and Decay}\label{Curvature&Decay}
In this section we explore the curvature of the 
Heisenberg spheres
 and present the proof of Proposition~\ref{alphaDecay}. 

\subsection{Curvature}\label{curvature}
We begin with some useful definitions summarized from Shakarchi and Stein's \cite[Chapter 8]{SandS4}.  
Given a local defining function $\varphi:\R^{n-1}\rightarrow \R$, 
the \textbf{principal curvatures} are the eigenvalues of the Hessian matrix \big($\nabla^2 \varphi=\{\frac{\partial^2\varphi}{\partial x_i \partial x_k}\}_{1\leq i,k \leq n-1}$\big), and the \textbf{Gaussian curvature} is the product of the principal curvatures.   
The surface of the Heisenberg sphere is given by 
$$\big((z_1^2 + \cdots + z_{2d}^2)^{\alpha/2} + |t|^{\frac{\alpha}{2}} \big)^{1/\alpha} = 1.$$
\vs

At the \textit{equator}, that is the intersection of $\partial{}B_1^{\alpha}$ and the hyperplane $t=0$, 
we use rotational symmetry to reduce matters to considering the curvature at the point $(1,0,\dots, 0)$ in $ \mathbb{R}^{2d+1}$.  
In particular, we solve for $z_1$ in terms of  $z_2, \dots, z_{2d}$ and $t$:
\begin{IEEEeqnarray}{cl}\label{solvez}
z_1 = \big[(1-|t|^{\frac{\alpha}{2}})^{2/\alpha} - (z_2^2 + \cdots + z_{2d}^2)\big]^{1/2}.
\end{IEEEeqnarray}
We then take derivatives with respect to $t$ and $z_k$, $k \neq 1$ and evaluate at $(0, \dots, 0) \in \mathbb{R}^{2d}$. 
Note that while the derivative with respect to $t$ does not exists for $\alpha=2$ and the second derivative with respect to $t$ does not exists when $\alpha =3$, we can differentiate twice with respect to all other variables.  
\vs

For $\alpha=4,$ the Hessian is a diagonal matrix with non-zero diagonal terms, and we conclude that the Gaussian curvature is non-zero.
For $\alpha>4$, the Hessian is also a diagonal matrix, but the second derivative in $t$ vanishes, i.e., one principal curvature vanishes at each point of the equator. 
Similarly, when $\al=2,3$, there are $2d-1= n-2$ non-vanishing curvatures.  
This can be seen by differentiating twice with respect to each variable in  $\{z_2, \dots, z_{2d}\}$ and arriving at a quantity that is not zero.  
More details these computations can be found in the thesis of the first listed author \cite[Chapter 4]{Campo}. 
\vs

At the \textit{poles}, the intersection of the set $\partial{}B_1^{\alpha}$ with the $t$-axis, we solve for $t>0$ in terms of $z \in \mathbb{R}^{2d}$:
\begin{IEEEeqnarray}{cl}\label{solvet}
t = \big(1-(z_1^2+ \cdots + z_{2d}^2)^{\alpha/2}\big)^{2/\alpha}.
\end{IEEEeqnarray}
Note it suffices to consider the north pole, $(0, \dots, 0, 1) \in \mathbb{R}^{2d+1}$, as the behavior at the south pole is equivalent by symmetry. 
By calculating the first and second derivatives in each $z_k$, it is evident that all the principal curvatures vanish for all $\alpha \geq 3$.
The exception is $\al=2$, where none of the principal curvatures vanish and the Gaussian curvature is non-vanishing.  
\vs

In summary, 
the poorest behavior for $\al=2$ occurs at each point of the equator where there are $n-2$ non-vanishing curvatures, 
and
the worst behavior for $\al\geq 3$ occurs at the poles where all principle curvatures vanish. 
\vs

At such points of maximal vanishing curvature, 
there are no results of which we are aware that would apply to give an estimate for the decay 
of the Fourier transform of the surface measures.
For this reason, a direct computation of the decay is necessary, the results of which are recorded in Proposition~\ref{alphaDecay}.  
On the other hand, in the instances when there is at least one non-vanishing curvature, 
we may rely on known results to establish the decay of the Fourier transform of the surface measure (see Theorem~\ref{Thm3.1} and Corollary~\ref{Cor3.2}).
\vs

\subsection{Surface Measure and Decay}\label{SurfMeasDecay} 
Before moving on to the proof of Proposition~\ref{alphaDecay}, we introduce the definition of surface measure and its Fourier transform, and we discuss two well-known decay estimates. We denote by $d\sigma$ the {\bf induced Lebesgue measure} on a $C^\infty$-hypersurface $M$, defined as follows: Given any continuous function $f$ on $M$ with compact support, 
\begin{equation}\label{inducedLebesgue}
\int_M f d\sigma = \lim_{\delta \rightarrow 0} \frac{1}{2\delta} \int_{d(x,M)<\delta} F dx,
\end{equation}
where $F$ is the extension of $f$ to a continuous function on a $\delta$-neighborhood of $M$.
\vs

The surface $M$ in $\mathbb{R}^n$ may be described locally up to a transformation as 
$\{(x',x_n) \in B : x_n = \varphi(x') \}$, 
where $B$ is a ball centered at the origin and $\varphi$ is a smooth function (for more details, see Chapter~7, Section~4 of \cite{SandS4}). 
Through this
local description of $M$, we may rewrite \eqref{inducedLebesgue} as
 \begin{equation}\label{rewriteILMeas} 
  \int_{\mathbb{R}^{n-1}} f(x',\varphi(x'))(1+ \abs{\nabla_{x'}\varphi}^2)^{1/2}dx'.
 \end{equation}
From this representation, if a measure $d\nu$ is of the form $d\nu = \psi d\sigma$, for some $C^{\infty}$ function $\psi$ with compact support, then $d\nu$ is a \textbf{surface measure} on $M$ with smooth density.
\vs

\par
We define the Fourier transform of such a surface measure by
\begin{equation}\label{FTdmuHatEq}
 \widehat{\nu}(\xi) = \int_M e^{-2\pi i x \cdot \xi}d\nu(x),
\end{equation}
which is bounded on $\mathbb{R}^n$ provided $d\nu$ is a finite measure.
\vs

Hence, letting $\tilde{\psi}(x')=\psi\big(x', \varphi(x')\big)\big(1+\abs{\nabla_{x'} \varphi}^2\big)^{\frac{1}{2}} \in C^\infty$ with compact support, we may rewrite \eqref{FTdmuHatEq} as
\begin{equation}\label{FTsigmaHatEq}
\widehat{\nu}(\xi) = \int_{\mathbb{R}^{n-1}}e^{-2\pi i(x',\varphi(x'))\cdot \xi}\tilde{\psi}(x')dx'.
\end{equation}
Note that ($spt(\tilde{\psi}) \subseteq spt(\psi)$).
\vs

The following theorem and corollary will be useful for computing the decay for non-polar points.
Theorem~\ref{Thm3.1} and Corollary~\ref{Cor3.2}, originally due to Hlawka \cite{Hlawka} and Littman \cite{Littman}, respectively, are well-known and appear in \cite{SandS4}.

\begin{theorem}[Decay estimate with non-vanishing Gaussian curvature]\label{Thm3.1}
Let $M \subseteq \mathbb{R}^n$ be a hypersurface with non-vanishing Gaussian curvature at all points $z \in spt(\nu)$. Then 
$$\big|\widehat{\nu}(\xi)\big| \lesssim \abs{\xi}^{-\frac{n-1}{2}}.$$
\end{theorem}

\begin{corollary}[Decay estimate when $m$ principal curvatures do not vanish]\label{Cor3.2}
Let $M \subseteq \mathbb{R}^n$ be a hypersurface such that at any point $z \in spt(\nu)$ at least $m$ principal curvatures do not vanish. $($This is equivalent to the statement that the Hessian of $\Phi$ has rank $m$.$)$ Then 
$$\big|\widehat{\nu}(\xi)\big| \lesssim \abs{\xi}^{-\frac{m}{2}}.$$
\end{corollary}

Referring to the curvature computations of Section~\ref{curvature}, 
we apply Theorem~\ref{Thm3.1} to estimate the decay at points with non-vanishing Gaussian curvature, and we apply Corollary~\ref{Cor3.2} with $m=n-2$ to estimate the decay at points with $n-2$ non-vanishing principal curvatures.  
Once we establish decay estimates in both regions, we take the minimal decay to obtain a universal decay estimate for the surface.

At the poles, we recall that all the principal curvatures vanish for $\alpha \geq 3$, so neither Theorem~\ref{Thm3.1} nor Corollary~\ref{Cor3.2} are applicable.
 Instead, we must estimate the decay for these points using our Proposition~\ref{alphaDecay}. 
We start with the proof for $n=3$ and $\alpha=4$.
\vs

\subsection{Proof of Proposition~\ref{alphaDecay}}
Following the discussion above, 
when $\al=2$, Corollary~\ref{Cor3.2}  is used to establish decay at points of the equator and Theorem~\ref{Thm3.1} is used to establish decay at the poles as well as at all remaining points.
\vs

It remains to prove Proposition~\ref{alphaDecay} at the poles for $\alpha \geq 3$. 
We focus on the north pole, $(0,0,1)$, as the estimate at the south pole follows through symmetry.
In short, we use a Taylor expansion to reformulate the arising phase function. 
Next, a polar change of variables is used to exploit the radiality of the function $\varphi$, at which point it is a matter of applying foundational oscillatory integral techniques.
 \vs 

For simplicity of presentation, we begin with the proof when $\al=4$ and then give the more general proof when $\al=3$ or $\al\geq 4$.  Note there is no inherent difference in the two proofs other than an extra level of generality.  

\subsection{Proof when $\al=4$}
\begin{proof}[Proof of Proposition~\ref{alphaDecay} when $\alpha=4$.]
 Given $n=3$ and $\alpha=4$, \eqref{solvet} becomes 
$t=\big(1-(z_1^2 + z_2^2)^2\big)^{\frac{1}{2}}$.
Set 
$$\varphi(z_1,z_2) = \big(1-(z_1^2 + z_2^2)^2\big)^{\frac{1}{2}},$$
and observe that $\varphi(0,0)=1$. Following the localized description of a hypersurface $M$ discussed above, we may describe the Heisenberg sphere near the north pole as $M= \{(z_1,z_2, t) \in \tilde{B}: t = \varphi(z_1,z_2)\}$. 
Let $\sigma_4$ denote the surface measure on $M$. 
Now, by  \eqref{FTsigmaHatEq},
\begin{equation}\label{sigma4Hat}
\widehat{\sigma}_4(\xi) 
= \int_{\mathbb{R}^2} 
e^{-2\pi i      (z_1, z_2, \varphi(z_1,z_2))    \cdot \xi} \tilde{\psi}(z_1,z_2) dz_1 dz_2, 
\end{equation}
where $\tilde{\psi}(z_1,z_2)=\psi\big(z', \varphi(z')\big)\big(1+\abs{\nabla_{z'} \varphi}^2\big)^{\frac{1}{2}} \in C_c^\infty$ for a smooth cut-off function $\psi$ supported on the unit ball. 
\vs

We proceed with the derivation of a 
Taylor approximation for $\varphi$ by considering $g(x)=\sqrt{x}$ and expanding around 1 using the first degree Taylor polynomial:
$$g(x)=1 + \frac{1}{2}(x-1) - \frac{ 1 }{8c^{\frac{3}{2}}} (x-1)^2 \ \text{for some} \ c \in (x,1).$$ 
\vs

Thus, we have 
\begin{align}
\varphi(z_1,z_2) = g(1-(z_1^2+z_2^2)^2) &= 1 + \frac{1}{2}\big(1-(z_1^2+z_2^2)^2-1\big) - \frac{1}{8c^{\frac{3}{2}}}\big(1-(z_1^2+z_2^2)^2-1\big)^2 \nonumber\\
&= 1- \frac{1}{2}(z_1^2+z_2^2)^2 - \frac{1}{8c^{\frac{3}{2}}}(z_1^2+z_2^2)^4  \label{tEst3D}.
\end{align}
\vs

Letting $\xi =(\xi_1,\xi_2,\xi_3) =(\tilde{\xi},\xi_3) \in \mathbb{R}^2 \textsf{ x } \mathbb{R}$
and plugging in the expression for $\varphi$ from \eqref{tEst3D}, 
\eqref{sigma4Hat} becomes
\begin{equation}
\widehat{\sigma}_4(\xi) \ = \int_{\mathbb{R}^2} e^{-2\pi i\left[(z_1,z_2) \cdot \tilde{\xi} + \left(1- \frac{1}{2}(z_1^2+z_2^2)^2 - \frac{1}{8c^{3/2}}(z_1^2+z_2^2)^4\right)\xi_3\right]} \tilde{\psi}(z_1,z_2) dz_1 dz_2.  \label{dmuQalpha43D}
\end{equation}
\vs

We use rotational symmetry near the poles to simplify to the case $\xi_2=0$, so that $|\tilde{\xi}|=|\xi_1|$ 
and $$\widehat{\sigma}_4(\xi) = \widehat{\sigma}_4( \xi_1, 0,\xi_3)   .$$
\vs

We now split our argument between two regions: 
The first is the critical region where $\xi$ may be normal to the pole, so $|\xi_3| \geq |\xi_1|$, and the second is its complementary region, $|\xi_3| < |\xi_1|$.
\vs

\noindent
\textbf{The case when }$\mathbf{|\xi_3| \geq |\xi_1| \!:}$
We first consider the region where $|\xi_3| \geq  |\xi_1|$.
Now, (\ref{dmuQalpha43D})  
with $\xi_2=0$
becomes
\begin{equation}
\int_{\mathbb{R}^2} e^{-2\pi i \xi_3\left[\frac{\xi_1}{\xi_3}z_1 + 1- \frac{1}{2}(z_1^2+z_2^2)^2 - \frac{1}{8c^{3/2}}(z_1^2+z_2^2)^4\right]} \tilde{\psi}(z_1,z_2) dz_1 dz_2. \label{dmuQalpha43Dsimp}
\end{equation}

Translating to polar coordinates and setting $\varphi(r,\theta)= r\frac{\xi_1}{\xi_3}\cos\theta 
+1-\frac{1}{2}r^4-\frac{1}{8c^{3/2}}r^8$, we see that
bounding
(\ref{dmuQalpha43Dsimp}) is  equivalent to bounding
\begin{equation}
 \!\int_0^{\frac{\pi}{2}} \! \int_0^1 \! e^{-2\pi i \xi_3\varphi(r,\theta)} r dr d\theta, 
\label{PolarDmuQAlph4}
\end{equation}
where we have used the compact support afforded by $\psi$ to restrict $r$ to the compact set $[0,1]$, and we reduced the integration in $\theta$ to $[0,\pi]$ using the periodicity of $\cos \theta$. 
We only address the integral over $\theta \in [0, \frac{\pi}{2}]$, as the integral over $\theta \in [\frac{\pi}{2},\pi]$ is handled using symmetry.
\vs

Fix $\theta \in (0, \frac{\pi}{2})$ and set
\begin{equation*}
\varphi_\theta(r)=\varphi(r, \theta)= r\frac{\xi_1}{\xi_3}\cos\theta +1-\frac{1}{2}r^4-\frac{1}{8c^{3/2}} r^8,
\end{equation*}
consequently, all derivatives will be taken with respect to $r$.
\vs

Observe that the first and second derivatives of $\varphi_\theta$ with respect to $r$ are
\begin{equation}\label{Alph4Phi'Phi''}
\varphi'_\theta(r)=\frac{\xi_1}{\xi_3}\cos\theta - 2r^3-  \frac{1}{c^{3/2}}r^7
\ \ \text{ and } \ \ 
\varphi''_\theta(r)=-6r^2- \frac{7}{c^{3/2}}r^6.
\end{equation}
Note that since $\varphi''_\theta(r) \le 0$, it follows that $\varphi'_\theta$ is monotone decreasing. 
Since $\theta$ is fixed in $(0, \frac{\pi}{2})$,
$\varphi'_\theta(r)$ has at most one zero in $(0,1)$, which we denote by $r_0$, 
satisfying 
\begin{equation}\label{Alph4r0_eq}
2r_0^3 + \frac{1}{c^{3/2}} r_0^7 = \frac{\xi_1}{\xi_3}\cos\theta.
\end{equation} 

We further note that if $r_0 \in (0,1)$, then $\varphi'_\theta$ is positive in $(0,r_0)$ and negative in $(r_0,1)$; else, $\varphi'_\theta$ does not change sign in $(0,1)$. Additionally, 
 $\frac{r}{\varphi'_\theta(r)}$ is monotone increasing.
\vs

We introduce an $\epsilon$-ball, 
$B_\epsilon(r_0)$, about the point $r_0$, giving two cases to consider: 
\begin{itemize}
\item[Case (A):] $|\xi_3|^{-\frac{1}{4}} > r_0$, where we 
let $\epsilon =|\xi_3|^{-\frac{1}{4}}$, and
\item[Case (B):] $|\xi_3|^{-\frac{1}{4}} \leq r_0$ where we 
have $\epsilon = \frac{1}{r_0} \cdot |\xi_3|^{-\frac{1}{2}}$.
\end{itemize}
\vs

Before we bound the expression in (\ref{PolarDmuQAlph4}) in each of these cases, we give some preliminaries important to both.
Away from $B_{\epsilon}(r_0)$, we utilize the following observation:
\begin{equation}
e^{-2\pi i\xi_3\varphi_\theta(r)}=\frac{i}{2\pi \xi_3\varphi'_\theta(r)} \frac{d}{dr}\big(e^{-2\pi i\xi_3\varphi_\theta(r)}\big). \label{eRelVarPhiTheta3D}
\end{equation}
It follows that, for any $a<b$ such that $|\varphi'_\theta(r)|\geq c'>0$ on $[a,b]$, 
\begin{align}
\abs{\int_{a}^b \! e^{-2\pi i \xi_3\varphi_\theta(r)} r dr} &= \frac{1}{2\pi |\xi_3|}\abs{\int_{a}^b \frac{d}{dr}\big(e^{-2\pi i \xi_3\varphi_\theta(r)}\big) \frac{r}{\varphi'_\theta(r)} dr} \nonumber\\
&= \frac{1}{2\pi |\xi_3|}\abs{\frac{re^{-2\pi i \xi_3\varphi_\theta(r)}}{\varphi'_\theta(r)}\Big|_a^b- \int_{a}^b \! e^{-2\pi i \xi_3\varphi_\theta(r)} \frac{d}{dr}\!\!\left(\!\frac{r}{\varphi'_\theta(r)}\!\right)\! dr} \ \ \ \ \ \ \ \ \nonumber\\
&\leq \frac{1}{2\pi |\xi_3|}\!\left[
\abs{
\frac{r }
{\varphi'_\theta(r)}    
   \Big|_a^b} 
+ 
\abs{\int_{a}^b \! e^{-2\pi i \xi_3\varphi_\theta(r)} \frac{d}{dr}\!\!\left(\!\frac{r}{\varphi'_\theta(r)}\!\right) \!dr}
\right] \!. 
\label{IntByPartsGen3D}
\end{align}
Furthermore, since 
$\varphi'_\theta$ is monotone and $\frac{d}{dr}\Big(\frac{r}{\varphi'_\theta(r)}\Big)$ is non-zero on $(a,b)$, the second expression in (\ref{IntByPartsGen3D}) is:
\begin{equation*}
\abs{
\int_{a}^b e^{-2\pi i \xi_3\varphi_\theta(r)} \frac{d}{dr}\!\left(\!\frac{r}{\varphi'_\theta(r)}\!\right) dr}
 \leq
  \int_a^b  \left|\frac{d}{dr}\!\left(\!\frac{r}{\varphi'_\theta(r)}\!\right)\right| dr 
  = \abs{\int_a^b \frac{d}{dr}\!\left(\!\frac{r}{\varphi'_\theta(r)}\!\right) dr},
\end{equation*}
which is the same as the first expression.
\vs

Combining the above, we conclude that 
\begin{align}
\abs{\int_{a}^b \! e^{-2\pi i \xi_3\varphi_\theta(r)} r dr} 
&\leq \frac{1}{\pi |\xi_3|}\!\left(
\abs{
\frac{r    
}
{\varphi'_\theta(r)}    
   \Big|_a^b} 
\right) \!. 
\label{IntByPartsGen3D_COMBINE}
\end{align}

This implies that a lower bound on $|\varphi'_\theta(r)|$ is essential. Observe that
\begin{align}
|\varphi'_\theta(r_0 \pm \epsilon)| \, &= |\varphi'_\theta(r_0 \pm \epsilon) - \varphi'_\theta(r_0)| \nonumber\\
&= \abs{\frac{\xi_1}{\xi_3}\cos\theta - 2(r_0 \pm \epsilon)^3 - c^{-\frac{3}{2}}(r_0 \pm \epsilon)^7 - \frac{\xi_1}{\xi_3}\cos\theta + 2r_0^3 + c^{-\frac{3}{2}}r_0^7} \nonumber\\
&= \abs{ - 2(r_0 \pm \epsilon)^3 - c^{-\frac{3}{2}}(r_0 \pm \epsilon)^7 + 2r_0^3 + c^{-\frac{3}{2}}r_0^7} \nonumber\\
& \gtrsim 
\begin{cases}
\epsilon^3, & \text{if } r_0<\epsilon\\
r_0^2\epsilon, & \text{if } r_0 >\epsilon. 
\label{phi'ThetaLowerAlph4}
\end{cases}
\end{align}
\vs

We now proceed to bound the inner integral of (\ref{PolarDmuQAlph4}), beginning with Case (A).
\begin{itemize}
\item[Case (A):] $|\xi_3|^{-\frac{1}{4}} > r_0$. Let $\epsilon = |\xi_3|^{-1/4}$, then 
\begin{equation*}
\abs{\int_{0}^1 e^{-2\pi i \xi_3\varphi_\theta(r)} r dr } \leq \ \abs{\int_{0}^\epsilon e^{-2\pi i \xi_3\varphi_\theta(r)} r dr} + \abs{\int_{\epsilon}^1 e^{-2\pi i \xi_3\varphi_\theta(r)} r dr},  \label{FirstInt1A}
\end{equation*}
and the first integral is simply bounded by $\epsilon^2=|\xi_3|^{-\frac{1}{2}}$.
\vs

For the second, 
we use (\ref{IntByPartsGen3D_COMBINE}) with $a=\epsilon$ and $b=1$ to bound 
\begin{equation*}
\abs{\int_{\epsilon}^1 e^{-2\pi i \xi_3\varphi_\theta(r)} r dr}
 \leq  \frac{1}{\pi |\xi_3|}\left(\abs{\frac{r   }{\varphi'_\theta(r)}\Big|_\epsilon^1} 
\right).
\end{equation*}
\vs

Furthermore,
(\ref{phi'ThetaLowerAlph4}) implies that
\begin{equation*}
\abs{\frac{r }{\varphi'_\theta(r)}\Big|_\epsilon^1} \lesssim \frac{1}{\epsilon^2}.
\end{equation*}
\vs

Hence, for Case (A), we have:
\begin{equation}
\abs{\int_{0}^1 e^{-2\pi i \xi_3\varphi_\theta(r)} r dr } \lesssim |\xi_3|^{-\frac{1}{2}} + \frac{1}{\epsilon^2|\xi_3|} \sim |\xi_3|^{-\frac{1}{2}}. \label{CaseABoundAlph4}
\end{equation}

\item[Case (B):] $|\xi_3|^{-\frac{1}{4}} \leq r_0$. Let $\epsilon =  \frac{1}{r_0}|\xi_3|^{-\frac{1}{2}}$, then $\abs{\int_{0}^1 e^{-2\pi i \xi_3\varphi_\theta(r)} r dr }$ is bounded by
\begin{equation*}
 \abs{\int_{0}^{r_0 -\epsilon} e^{-2\pi i \xi_3\varphi_\theta(r)} r dr} + \abs{\int_{r_0-\epsilon}^{r_0 +\epsilon}  e^{-2\pi i \xi_3\varphi_\theta(r)} r dr} + \abs{\int_{r_0 + \epsilon}^1e^{-2\pi i \xi_3\varphi_\theta(r)} r dr}, \label{FirstInt1B}
\end{equation*}
and the second integral is simply bounded by $r_0\epsilon=|\xi_3|^{-1/2}$, since $r_0 \geq \epsilon$.\footnote{In the boundary case, where $r_0=\epsilon$, the first integral is simply zero and the second becomes an integral to $2\epsilon$, which produces no substantive change.} 
\vs

Applying (\ref{IntByPartsGen3D_COMBINE}) with $a=0$ and $b=r_0-\epsilon$, it follows that 
\begin{equation*}
\abs{\int_0^{r_0-\epsilon} \! e^{-2\pi i \xi_3\varphi_\theta(r)} r dr} \leq  
\frac{1}{\pi |\xi_3|} 
\!\left(
\abs{\frac{r}{\varphi'_\theta(r)}\Big|_0^{r_0-\epsilon}} 
\right) \!. \
\end{equation*}
Since $|\varphi'_\theta(r_0-\epsilon)| \gtrsim r_0^2\epsilon$ by (\ref{phi'ThetaLowerAlph4}), 
plugging in $\epsilon =  \frac{1}{r_0}|\xi_3|^{-\frac{1}{2}}$, it follows that this is further bounded above by 
\begin{equation*}
\frac{1}{|\xi_3|r_0\epsilon} = \frac{r_0|\xi_3|^{1/2}}{|\xi_3|r_0} = |\xi_3|^{-1/2}.
\end{equation*} 
\vs

The bound on the third integral is the same and its proof proceeds similarly. 
\end{itemize}
\vs

Therefore, in both Cases (A) and (B), 
\begin{equation}
\abs{\int_{0}^1 e^{-2\pi i \xi_3 \varphi_\theta(r)} r dr} \lesssim |\xi_3|^{-1/2}. \label{CaseBBoundAlph4}
\end{equation}

Integrating (\ref{CaseABoundAlph4}) and (\ref{CaseBBoundAlph4}) in $\theta$, it follows that 
\begin{align*}
\abs{\widehat{\sigma}_4(\xi)} \, &\sim \abs{\int_0^{\frac{\pi}{2}} \int_0^1 e^{-2\pi i \xi_3\varphi_\theta(r)} r dr d\theta } \\
&\lesssim \int_0^{\frac{\pi}{2}} \frac{1}{|\xi_3|^{1/2}} d\theta \\
&= \frac{\pi}{2|\xi_3|^{1/2}}.
\end{align*}
Recalling that $\abs{\xi_3} \geq |\xi_1|$, we conclude that
\begin{equation}\label{3DAlpha4decay1}
\abs{\widehat{\sigma}_4(\xi)} \lesssim \abs{\xi}^{-\frac{1}{2}}.
\end{equation}
\vs
%
%
\noindent
\textbf{The case when }$\mathbf{|\xi_3|<|\xi_1| \!:}$
Now, in the region where $|\xi_1|>|\xi_3|$, it follows from (\ref{dmuQalpha43D}) that  
\begin{align}
\widehat{\sigma}_4(\xi)=  \int_{\mathbb{R}^2} e^{-2\pi i \xi_1\big[z_1 + \frac{\xi_3}{\xi_1}\big(1- \frac{1}{2}(z_1^2+z_2^2)^2 - \frac{1}{8c^{3/2}}(z_1^2+z_2^2)^4\big)\big]} \tilde{\psi}(z_1,z_2) dz_1 dz_2. \label{dmuQalpha43Dsimp2}
\end{align}
\vs


Translating this into polar coordinates and using the compact support afforded by $\psi$ to restrict $r$ to the compact set $[0,1]$, we set
\begin{equation*}
\varphi(r,\theta)= r\cos\theta + \frac{\xi_3}{\xi_1}\big(1-\frac{1}{2}r^4-\frac{1}{8c^{3/2}}r^8\big). 
\end{equation*}
In this way, bounding (\ref{dmuQalpha43Dsimp2}) is equivalent to bounding
\begin{align}
\int_0^1 \int_0^{\pi/2}  e^{-2\pi i \xi_1\varphi(r,\theta)} r d\theta dr
\label{xi1Polar3D}
\end{align}
\vs

Fixing $r$ in $(0,1]$, we denote
\begin{equation*}
\varphi_r(\theta) = \varphi(r,\theta)= r\cos\theta + \frac{\xi_3}{\xi_1}\big(1-\frac{1}{2}r^4-\frac{1}{8c^{3/2}}r^8\big).
\end{equation*}

Now observe that the derivative of $\varphi_r$ with respect to $\theta$ is
\begin{equation*}
\varphi_r'(\theta) = -r\sin \theta, 
\end{equation*}
which has a single zero at $\theta=0$. 

\vs

Let $B_\epsilon$ denote the interval $(0,\epsilon)$. Then the inner integral in \eqref{xi1Polar3D} becomes
\begin{align*}
\bigg| \int_0^{\pi/2} \! e^{-2\pi i \xi_1\varphi_r(\theta)}  d\theta \bigg| \nonumber
&=
\bigg|
 \int_{B_\epsilon} e^{-2\pi i \xi_1\varphi_r(\theta)}  r d\theta  +
\int_\epsilon^{{\pi/2}} e^{-2\pi i \xi_1\varphi_r(\theta)}  r d\theta \bigg| \\
 &\leq 
  \int_{B_\epsilon} r d\theta  \,\,\,+\,\,\, 
\bigg| \int_\epsilon^{\pi/2} e^{-2\pi i \xi_1\varphi_r(\theta)}  r d\theta \bigg|.
 \end{align*}
 
 The first integral is simply a constant times $\epsilon$.
For the second,
utilize the van der Corput Lemma, which appears in most texts on the topic of stationary phase, such as \cite[Chapter 8]{Stein} and \cite[Chapter 14]{Mattila15}.  

\begin{lemma}[van der Corput]\label{vanDerCorput}
Given $\phi : [a,b] \rightarrow \mathbb{R}$ smooth 
such that for all $x \in (a,b)$, $|\phi^{(k)}(x)| \geq s>0$, then 
$$\bigg|\int_a^b e^{i \xi \phi(x)} dx\bigg| \leq c_k|s\xi|^{-\frac{1}{k}},$$
if $\phi'$ is monotone or $k \geq 2$.
\end{lemma}
This lemma is generally written with  $|\phi^{(k)}(x)| \geq 1$, but we will use this more general formulation.
\vs

Now, observe that  $\varphi_r'(\theta) = -r\sin{\theta} $ is monotone and $|\varphi_r'(\theta)| \geq \frac{1}{2}r\epsilon>0$ on $\left[\epsilon, \frac{\pi}{2}\right]$ since $r>0$. Thus, applying van der Corput
(Lemma~\ref{vanDerCorput}) to the second integral, 
\begin{align*}
 \abs{ \int_\epsilon^{\frac{\pi}{2}} e^{-2\pi i \xi_1 \varphi_r(\theta)}   r d\theta } 
&= r  \abs{ \int_\epsilon^{\frac{\pi}{2}} e^{-2\pi i \xi_1 \varphi_r(\theta)}   d\theta }  \nonumber\\
 &\lesssim  r |r\epsilon \xi_1 |^{-1} 
 \nonumber\\
 &= \frac{1}{|\epsilon\xi_1|}.
 \end{align*} 
 \vs

Hence, when $|\xi_1|>|\xi_3|$, we have
\begin{equation*}
\bigg| \!\int_0^\pi \! e^{-2\pi i \xi_1\varphi_r(\theta)} r d\theta \bigg| 
\lesssim  \epsilon + \frac{1}{|\xi_1|} \cdot \frac{1}{\epsilon} 
=\frac{1}{|\xi_1|^{\frac{1}{2}}},
\end{equation*}
choosing $\epsilon =  |\xi_1|^{-\frac{1}{2}}$.
Integrating in $r$ in $[0,1]$ does not change this estimate and completes the bound for \eqref{xi1Polar3D}. 
This concludes the proof when $\al=4$. 
\end{proof}
\vs

\subsection{Proof when $\al=3$ or $\al\geq 4$}
We now proceed with the proof for a general $\alpha=3$ or $\alpha>4$ in 3-dimensions.
\vs
\begin{proof}[Proof of Proposition~\ref{alphaDecay} when $n=3$ and $\alpha = 3$ or $\alpha >4$]
We again focus on the north pole, $(0,0,1)$, and \eqref{solvet} becomes
$t=\big(1-(z_1^2 + z_2^2)^{\alpha/2}\big)^{2/\alpha}$. 
\vs

Set $\varphi(z_1,z_2) = \big(1-(z_1^2 + z_2^2)^{\alpha/2}\big)^{2/\alpha}$ and observe that we again have $\varphi(0,0)=1$. Proceeding as in the case where $\alpha=4$, 
we use \eqref{FTsigmaHatEq} to formulate the Fourier transform of the surface measure as
\begin{equation}\label{sigmaHat3D}
\widehat{\sigma}_\alpha(\xi) 
= \int_{\mathbb{R}^2} 
e^{-2\pi i    (z_1, z_2, \varphi(z_1,z_2))    \cdot \xi} \tilde{\psi}(z_1,z_2) dz_1 dz_2, 
\end{equation}
where $\tilde{\psi}(z_1,z_2)=\psi\big(z', \varphi(z')\big)\big(1+\abs{\nabla_{z'} \varphi}^2\big)^{\frac{1}{2}} \in C_c^\infty$ and $\psi(z_1,z_2)$ is again a smooth bump function supported on the unit ball.
\vs

Now, we estimate $\varphi$ by its Taylor expansion about 1.
Let $g(x)=x^{2/\alpha}$. Then the first order Taylor expansion of $g(x)$ about 1 is 
$$g(x) = 1 + \frac{2}{\alpha}(x-1) - \frac{c_\alpha(\alpha-2)}{\alpha^2}(x-1)^2$$
where $c_\alpha = c^{\frac{2-\alpha}{\alpha^2}}$ for some $\ c \in (x,1).$
Recognizing that $\varphi(z_1, z_{2}) = g(1-(z_1^2 + z_{2}^2)^{\frac{\alpha}{2}})$, we have 
\begin{align}
\varphi(z_1, z_{2}) \
&= 1+ \frac{2}{\alpha}(1-(z_1^2 + z_{2}^2)^{\frac{\alpha}{2}}-1) -\frac{c_\alpha(\alpha-2)}{\alpha^2}(1-(z_1^2 + z_{2}^2)^{\frac{\alpha}{2}}-1)^2 \nonumber\\
&= 1- \frac{2}{\alpha}(z_1^2 + z_{2}^2)^{\frac{\alpha}{2}} - \frac{c_\alpha(\alpha-2)}{\alpha^2}(z_1^2 + z_{2}^2)^{\alpha}.\label{tEst3Dgen}
\end{align}
\vs

Proceeding as in the $\al=4$ case, we rewrite the phase function in \eqref{sigmaHat3D} using \eqref{tEst3Dgen}, so that letting
 $\xi =(\tilde{\xi},\xi_3) \in \mathbb{R}^2 \ \textsf{x} \ \mathbb{R}$ gives
\begin{equation*}
\widehat{\sigma}_\alpha(\xi) = \int_{\mathbb{R}^2} e^{-2\pi i \left[(z_1,z_2) \cdot \tilde{\xi} + \left(1- \frac{2}{\alpha}(z_1^2 + z_{2}^2)^{\frac{\alpha}{2}} - \frac{c_\alpha(\alpha-2)}{\alpha^2}(z_1^2 + z_{2}^2)^{\alpha}\right)\xi_3\right]} \tilde{\psi}(z_1,z_2) dz_1 dz_2. \label{dmuQalpha3D}
\end{equation*}
\vs

As before, we simplify to the case $\xi_2=0$ via rotational symmetry, so that $|\tilde{\xi}| = |\xi_1|$ and 
$$\widehat{\sigma}_\alpha(\xi) = \widehat{\sigma}_\alpha( \xi_1, 0,\xi_3) .$$ 
\vs

\noindent
\textbf{The case when }$\mathbf{|\xi_3| \geq |\xi_1| \!:}$
Considering first the region where $|\xi_3| \geq |\xi_1|$, we have: 
\begin{align}
\widehat{\sigma}_\alpha(\xi) &= \int_{\mathbb{R}^2} e^{-2\pi i\xi_3 \big[\frac{\xi_1}{\xi_3}z_1 + 1- \frac{2}{\alpha}(z_1^2 + z_{2}^2)^{\frac{\alpha}{2}} - \frac{c_\alpha(\alpha-2)}{\alpha^2}(z_1^2 + z_{2}^2)^{\alpha}\big]} \tilde{\psi}(z_1,z_2) dz_1 dz_2. \label{dmuQalpha3D3Dom}
\end{align}
\vs

Translating into polar coordinates, set  
\begin{equation}\label{phi_eq} 
\varphi(r,\theta)=r\frac{\xi_1}{\xi_3}\cos\theta+1-\frac{2}{\alpha}r^{\alpha}-\frac{c_\alpha(\alpha-2)}{\alpha^2}r^{2\alpha},
\end{equation}
and restrict $r$ to the compact set $[0,1]$ using the compact support afforded by $\psi$. Then we may again reduce the integration in $\theta$ to $[0,\pi]$ using the periodicity of $\cos \theta$, so that
to bound (\ref{dmuQalpha3D3Dom}) it suffices to bound 
\begin{equation}\label{dmuQalpha3D3DomPol}
\int_0^{\frac{\pi}{2}} \! \int_0^1 e^{-2\pi i \xi_3\varphi_\theta(r)} r dr d\theta,
\end{equation}
since the integral over $\theta \in [\frac{\pi}{2},\pi]$ is handled via symmetry. 
\vs

Fixing $\theta$ in $\left(0,\frac{\pi}{2}\right)$, we denote the phase function by
\begin{equation*}
\varphi_\theta(r) = \varphi(r,\theta)= r\frac{\xi_1}{\xi_3}\cos\theta+1-\frac{2}{\alpha}r^{\alpha}-\frac{c_\alpha(\alpha-2)}{\alpha^2}r^{2\alpha},
\end{equation*}
consequently, all derivatives will be taken with respect to $r$.

Now observe that the first and second derivatives of $\varphi_\theta$ with respect to $r$ are
\begin{equation}\label{phi_deriv}
\varphi'_\theta(r)=\frac{\xi_1}{\xi_3}\cos\theta -2r^{\alpha-1}-\frac{2c_\alpha(\alpha-2)}{\alpha}r^{2\alpha-1}
\end{equation}
 and 
 \begin{equation}\label{phi_second_deriv}
 \varphi''_\theta(r)=-2(\alpha-1)r^{\alpha-2}-\frac{2c_\alpha(\alpha-2)(2\alpha-1)}{\alpha}r^{2\alpha-2}.
 \end{equation}
We further note that 
since $\varphi_\theta''(r) < 0$ for all $r\in (0,1)$, it follows that $\frac{r}{\varphi'_\theta(r)}$ is monotone increasing. 
Since $\theta$ is fixed in $\left(0,\frac{\pi}{2}\right)$, 
$\varphi'_\theta$ has at most one zero at the point $r_0 \in (0,1)$ 
where 
\begin{equation*}
2r_0^{\alpha-1} + \frac{2c_\alpha(\alpha-2)}{\alpha}r_0^{2\alpha-1} = \frac{\xi_1}{\xi_3}\cos\theta.
\end{equation*} 
Observe that $\varphi'_\theta$ is therefore positive on $(0,r_0)$ and negative on $(r_0,1)$. 
\vs 

We introduce an $\epsilon$-ball, 
$B_\epsilon(r_0)$, about the point $r_0$, giving two cases to consider: 
\begin{itemize}
\item[Case (A):] $|\xi_3|^{-\frac{1}{\alpha}} > r_0$, where we set $\epsilon$-ball ($\epsilon =|\xi_3|^{-\frac{1}{\alpha}}$), and
\item[Case (B):] $|\xi_3|^{-\frac{1}{\alpha}} \leq r_0$ where we set $\epsilon = r_0^{1-\frac{\alpha}{2}} \cdot |\xi_3|^{-\frac{1}{2}}$. 
\end{itemize}
\vs

We follow the same approach as in the case of $\alpha=4$ to bound \eqref{dmuQalpha3D3DomPol}, first
observing that since $r \leq 1$,
\begin{align}
|&\varphi'_\theta (r_0 \pm \epsilon)| \nonumber\\
&= |\varphi'_\theta(r_0 \pm \epsilon) - \varphi'_\theta(r_0)| \nonumber\\
&= \abs{\frac{\xi_1}{\xi_3}\!\cos\theta \!-\! 2(r_0 \pm \epsilon)^{\alpha-1} \!\!- \! \frac{2c_\alpha(\alpha-2)}{\alpha}(r_0 \pm \epsilon)^{2\alpha-1}\!\! - \! \frac{\xi_1}{\xi_3}\!\cos\theta \!+\! 2r_0^{\alpha-1} \!\!+ \! \frac{2c_\alpha(\alpha-2)}{\alpha}r_0^{2\alpha-1}} \nonumber\\
&= \abs{-2(\alpha-1)r_0^{\alpha-2}(\pm \epsilon) - \sum_{k=2}^{\alpha-1} {\alpha-1 \choose k} r_0^{\alpha-1-k}(\pm \epsilon)^k - \sum_{k=1}^{2\alpha-1} {2\alpha-1 \choose k} r_0^{2\alpha-1-k}(\pm \epsilon)^k} \nonumber\\
& \gtrsim 
\begin{cases}
\epsilon^{\alpha-1}, & \text{if } r_0<\epsilon \\
r_0^{\alpha-2}\epsilon, & \text{if } r_0 >\epsilon. \label{phi'ThetaLower3D}
\end{cases}
\end{align}
\vs

We begin with the bound of the inner integral of \eqref{dmuQalpha3D3DomPol} in the case that $r_0 \leq |\xi_3|^{-\frac{1}{\alpha}}$:

\begin{itemize}
\item[Case (A):] $|\xi_3|^{-\frac{1}{\alpha}} > r_0$. Let $\epsilon = |\xi_3|^{-1/\alpha}$, then we have 
\begin{equation*}
\abs{\int_{0}^1 e^{-2\pi i \xi_3\varphi_\theta(r)} r dr } \leq \ \abs{\int_{0}^\epsilon e^{-2\pi i \xi_3\varphi_\theta(r)} r dr} + \abs{\int_{\epsilon}^1 e^{-2\pi i \xi_3\varphi_\theta(r)} r dr},  \label{FirstInt1A3D}
\end{equation*}
where the first integral is simply bounded by $\eps^2 = |\xi_3|^{-\frac{2}{\alpha}}$.
\vs

For the second, we use \eqref{IntByPartsGen3D_COMBINE} as before with $a=\epsilon$ and $b=1$, since $\frac{d}{dr}\!\left(\frac{r}{\varphi'_\theta(r)}\right)$ does not change sign on $(\epsilon,1)$. This gives
\begin{equation*}
\abs{\int_{\epsilon}^1 e^{-2\pi i \xi_3\varphi_\theta(r)} r dr} 
\leq   \frac{1}{\pi |\xi_3|}\Bigg[\bigg|\frac{r}{\varphi'_\theta(r)}\Big|_\epsilon^1\bigg| \Bigg].
\end{equation*}
Furthermore, since $\frac{r}{\varphi'_\theta(r)}<0$ and increasing on $(\eps,1)$, 
(\ref{phi'ThetaLower3D}) implies that
\begin{equation*}
\abs{\frac{r}{\varphi'_\theta(r)}\Big|_\epsilon^1} \lesssim \frac{1}{\epsilon^{\alpha-2}}.
\end{equation*}
\vs

Hence, for Case (A), plugging in $\epsilon = |\xi_3|^{-1/\alpha}$, we have:
\begin{equation}
\abs{\int_{0}^1 e^{-2\pi i \xi_3\varphi_\theta(r)} r dr } 
\lesssim \frac{1}{|\xi_3|} \cdot \epsilon^{-(\alpha-2)}
=\frac{1}{|\xi_3|} \cdot |\xi_3|^{\frac{\alpha-2}{\alpha}} 
= |\xi_3|^{-\frac{2}{\alpha}}. \label{CaseABound3D}
\end{equation}

\item[Case (B):] $|\xi_3|^{-\frac{1}{\alpha}} \leq r_0$.
Let $\epsilon =  r_0^{1-\frac{\alpha}{2}}|\xi_3|^{-\frac{1}{2}}$, then we have $\abs{\int_{0}^1 e^{-2\pi i \xi_3\varphi_\theta(r)} r dr }$ bounded by
\begin{equation}\label{IntSumB3D}
\abs{\int_{0}^{r_0-\epsilon} e^{-2\pi i \xi_3\varphi_\theta(r)} r dr } + \abs{\int_{r_0-\epsilon}^{r_0 +\epsilon} e^{-2\pi i \xi_3\varphi_\theta(r)} r dr} + \abs{\int_{r_0 + \epsilon}^1 e^{-2\pi i \xi_3\varphi_\theta(r)} r dr},  
\end{equation}
where the second integral is simply bounded by $r_0\epsilon=|\xi_3|^{-\frac{1}{2}}r_0^{2-\frac{\alpha}{2}}$, since $r_0 \geq \epsilon$.\footnote{In the boundary case, where $r_0=|\xi_3|^{-\frac{1}{\alpha}}$, so $r_0=\epsilon$, the first integral is simply zero and the second becomes an integral to $2\epsilon$, which produces no substantive change.} This is further bounded by $|\xi_3|^{-\frac{2}{\alpha}}$ when $\alpha >4$, since $\frac{1}{r_0}\leq|\xi_3|^{\frac{1}{\alpha}}$,
and by $|\xi_3|^{-\frac 1 2} $ when $\alpha=3$, since $r_0 \leq 1$.
\vs

For the first integral, recalling from our preliminaries
that $\frac{d}{dr}\big(\frac{r}{\varphi'_\theta(r)}\big)$ does not change sign on $(0, r_0-\epsilon)$, we know from \eqref{IntByPartsGen3D_COMBINE} that, for $a = 0$ and $b=r_0-\epsilon$, 
\begin{equation*}
\abs{\int_0^{r_0-\epsilon} \! e^{-2\pi i \xi_3\varphi_\theta(r)} r dr}
\leq  \frac{1}{\pi |\xi_3|} \!\Bigg[\bigg|\frac{r}{\varphi'_\theta(r)}\Big|_0^{r_0-\epsilon}\bigg| \Bigg].
\end{equation*}
\vs

Note that we again have $\frac{d}{dr}\big(\frac{r}{\varphi'_\theta(r)}\big)$ increasing and negative, with $|\varphi'_\theta(r_0-\epsilon)| \gtrsim r_0^{\alpha-2}\epsilon$ by (\ref{phi'ThetaLower3D}). Plugging in $\epsilon =  r_0^{1-\frac{\alpha}{2}}|\xi_3|^{-\frac{1}{2}}$, it follows that this is further bounded above by 
a constant multiple of 
\begin{equation}\label{FirstInt1B3D}
\frac{1}{|\xi_3|r_0^{\alpha-3}\epsilon} = 
|\xi_3|^{-\frac{1}{2}}r_0^{-\frac{\alpha-4}{2}} \leq |\xi_3|^{-\frac{2}{\alpha}},
\end{equation} 
when $\alpha>4$, since 
$\frac{1}{r_0} \leq |\xi_3|^{\frac{1}{\alpha}}$, and by $|\xi_3|^{-\frac 1 2} $ when $\al=3$.
\vs

The process to bound the third integral is similar and yields the same bound. 
\end{itemize}
\vs

Therefore, combining Case (A) and Case (B), we have
\begin{equation}
\abs{\int_{0}^1 e^{-2\pi i \xi_3 \varphi_\theta(r)} r dr} \lesssim |\xi_3|^{-\frac{2}{\alpha}}, \text{ for } \alpha>4, \label{CaseABBound3D}
\end{equation}
and
\begin{equation}
\abs{\int_{0}^1 e^{-2\pi i \xi_3 \varphi_\theta(r)} r dr} \lesssim |\xi_3|^{-\frac 1 2 }, \text{ for } \alpha=3, \label{CaseABBound3Dalpha3}
\end{equation}
where we have used the fact that 
$\frac{1}{2} < \frac{2}{3}=\frac{2}{\alpha}$ when $\alpha=3$.
\vs

Integrating 
(\ref{CaseABBound3D}) and 
(\ref{CaseABBound3Dalpha3})
in $\theta$, it follows that
$$
\abs{\widehat{\sigma}_\alpha(\xi)} \, = 2\abs{\int_0^\pi \int_0^1 e^{-2\pi i \xi_3\varphi_\theta(r)} r dr d\theta } 
\lesssim 
\begin{cases}
|\xi_3|^{-2/\al} \text{ if } \al >4\\
|\xi_3|^{-1/2} \text{ if } \al =3.\\
\end{cases}
$$
\vs

As we are in the case when $\abs{\xi_3} \geq |\xi_1|$, we conclude that 
$$
\abs{\widehat{\sigma}_\alpha(\xi)} 
\lesssim 
\begin{cases}
|\xi|^{-2/\al} \text{ if } \al >4\\
|\xi|^{-1/2} \text{ if } \al =3.\\
\end{cases}
$$
\vs

%
%
\noindent
\textbf{The case when }$\mathbf{|\xi_3|<|\xi_1| \!:}$
In
 the region where $|\xi_1|>|\xi_3|$, we have 
\begin{align}
\widehat{\sigma}_\alpha(\xi) =  \int_{\mathbb{R}^2} e^{-2\pi i \xi_1\left[z_1 + \frac{\xi_3}{\xi_1}\left(1- \frac{2}{\alpha}(z_1^2+z_2^2)^{\frac{\alpha}{2}} - \frac{c_\alpha(\alpha-2)}{\alpha^2}(z_1^2+z_2^2)^\alpha\right)\right]} \tilde{\psi}(z_1,z_2) dz_1 dz_2. \label{dmuQalpha3Dsimp2}
\end{align}
\vs


We translate this into polar coordinates, with phase function 
$\varphi(r,\theta)= r\cos\theta + \frac{\xi_3}{\xi_1}\big(1-\frac{2}{\alpha}r^\alpha-\frac{c_\alpha(\alpha-2)}{\alpha^2}r^{2\alpha}\big)$, 
and apply Fubini's Theorem so that (\ref{dmuQalpha3Dsimp2}) is 
\begin{align*}
\int_{S^1} \int_0^1 e^{-2\pi i \xi_1\varphi(r,\theta)} r dr d\theta 
= \int_0^1 \int_{S^1}  e^{-2\pi i \xi_1\varphi(r,\theta)} r d\theta dr. \label{xi1PolarAlph3D}
\end{align*}
Note that, as before, we have used the compact support afforded by $\psi$ to restrict $r$ to the compact set $[0,1]$.
\vs

Fix $r \in (0,1]$ and observe that the derivative with respect to $\theta$ of $\varphi$
is $-r\sin \theta$, as in the $\alpha=4$ case.
Following the argument immediately under \eqref{xi1Polar3D}, 
we conclude that when $|\xi_1|>|\xi_3|$, we have
\begin{equation*}
\left| \widehat{\sigma}_\alpha(\xi)\right| \lesssim |\xi|^{-\frac{1}{2}},
\end{equation*}
when $\al=3$ or $\al>4$.  
\vs

Therefore, 
combining the cases when 
$\mathbf{|\xi_3|>|\xi_1|}$ and $\mathbf{|\xi_3|<|\xi_1|}$, we conclude
\begin{equation*}
\abs{\widehat{\sigma}_\alpha(\xi)} \lesssim 
\begin{cases}
|\xi|^{-\frac{2}{\alpha}}, & \text{if } \alpha > 4 \\
|\xi|^{-\frac{1}{2}}, & \text{if } \alpha = 3.
\end{cases}
\end{equation*}
\end{proof}

\vskip.125in 
\section{Energy Estimate}\label{GeoSum}
In this section, we give the proof of Proposition \ref{energy}. 
In particular, we establish that
given $\al\geq 2$, $\tau \in \left(a, \frac{(n-1)a}{n-1-\gamma(\al)}\right]$ and $t =n - \gamma(\alpha)$, then 
\begin{equation}\label{strong_energy}
\iint \abs{x-y}^{-t} d\mu_q(x)d\mu_q(y) \lesssim 1,
\end{equation}
where $\ga(\al) $ is defined as in Proposition \ref{alphaDecay}.
\vs

We present only the proof of Proposition \ref{energy} when $n=3$, and refer to the thesis of the first listed author, \cite{Campo}, for the details of the proof when $n>3$.   The reasoning for this is that the proof for $n>3$ is rather lengthy and simply recycles the ideas presented here in the $n=3$ case.
\vs 
For computational purposes, it will be convenient to note that when $n=3$, 
\begin{equation}\label{tbound}
2.5\le t<3
\end{equation}
 for all $\alpha \geq2$. 
Furthermore, with $\tau$ and $t$ as above, 
when $n=3$, 
\begin{equation}\label{taubound3D}
\tau \leq \frac{(n-1)a}{n-1-\gamma(\al)} 
= \frac{2a}{t-1}.
\end{equation}
%


\subsection{Preliminaries and strategy for the proof of the energy estimate}\label{energyStrategy}

We use the definition of $\mu_q$ from Definition~\ref{defMuQalpha} to
 expand out the energy integral on the left-hand-side of \eqref{strong_energy}.

First, we briefly recall the set-up from Section \ref{reductSec} in the case $n=3$.  
$R_\tau$ denotes the $q^{-\tau} \ \mathsf{x} \ q^{-\tau} \ \mathsf{x} \ q^{-a-\tau}$ rectangular box centered at the origin. 
The set $E_q$, defined in \eqref{EqDef}, is a subset of the unit square made up of $q^3$ shifted copies of the rectangle $R_\tau$.  
For $b\in \mathbb{R}^3$, let $R_b = R_\tau+b$, so that $b$ is the center of this shifted rectangle. 
$V_R$ denotes the volume of such a rectangle so that $V_R = q^{-a-3\tau}$.
$N_R=q^3$ denotes the number rectangles making up $E_q$. 
Furthermore, $\abs{E_q} = N_R \cdot V_R$, where $|E_q|$ denotes the volume of $E_q$. 
\vs

Recall, $L_{3,q}$ denotes the truncated lattice: 
$$L_{3,q} = \{ b=(b_1, b_2,b_3) \in \mathbb{Z}^3: 
0< b_1\le q^a, \,\, 0< b_2\le q^a, \,\, 0< b_3\le q^{2a}\},$$
and observe, since $a=3/4$, that  $|L_{3,q}|  = q^3= N_R$. 
With this notation, we have  $$E_q = \bigcup_{b \in L_{3,q}} \left\{ R_\tau + b \right\}.$$
The left-hand-side of \eqref{strong_energy} is now equivalent to 
\begin{IEEEeqnarray}{rl}
\iint \abs{x-y}^{-t} d\mu_q(x) d\mu_q(y) 
\sim& \left(\frac{1}{\abs{E_q}}\right)^{\!2} \sum_{b,b' \in  L_{a,q}}  
\int_{R_b} \int_{R_{b'}}|x-y|^{-t} dy dx \nonumber   \\
 =& \left(\frac{1}{\abs{E_q}}\right) \sum_{b \in  L_{a,q}}  
\int_{R_b}    J(b,b') dx    
\label{OscIntkt}
\end{IEEEeqnarray}

where $b=(b_1,b_2,b_3)$, $b'=(b'_1,b'_2,b'_3)$, and
\begin{equation}\label{Jb}
J(b,b') = \left(\frac{1}{\abs{E_q}}\right)  \sum_{b' \in  L_{a,q}}    \int_{R_{b'}}|x-y|^{-t} dy.
\end{equation}
Matters are now reduced to showing that for each $b,b' \in L_{a,q}$,
\begin{equation}\label{reductionkt}
J(b,b')\lesssim 1,
\end{equation}
when $\tau \in \left(a, \frac{2a}{t-1}\right]$. Indeed, upon 
establishing \eqref{reductionkt}, the expression in \eqref{OscIntkt} is bounded  by
$$
  \left(\frac{1}{\abs{E_q}}\right) \sum_{b \in  L_{a,q}}  
\int_{R_b}    J(b,b') dx    
\lesssim 
  \left(\frac{1}{\abs{E_q}}\right) \sum_{b \in  L_{a,q}}  
V_R 
=   \left(\frac{1}{\abs{E_q}}\right) N_R \cdot V_R = 1,
$$
and Proposition~\ref{energy} will be proved.  
\vs

In order to establish \eqref{reductionkt} for an arbitrary choice of $b=(b_1,b_2,b_3) ,b'= (b_1',b_2',b_3') \in L_{a,q}$, we conduct a case analysis on $b,b'$: 
\begin{description}
\item[(1)] $b=b'$ ($x$ and $y$ are in the same rectangle) or
\item[(2)] $b \neq b'$ ($x$ and $y$ are in different rectangles).
\end{description}
The second case, when $b\neq b'$, is further divided into subcases depending on the proximity of the rectangles. 
We present the various sub-cases in order of increasing level of complexity.
\vs

 In Section~\ref{b=b'3D}, we address the first case where $b=b'$ using a dyadic shell argument that considers the intersection of the shells with the rectangle. 
In Sections ~\ref{b3=b'3ALLneq3D} through ~\ref{b3NEQb'33D}, we address the second case when $b \neq b'$. 
Specifically, 
in Section~\ref{b3=b'3ALLneq3D}, we address the cases where
either $b_1\neq b_1'$ or $b_2 \neq b_2'$.  In these cases, we say that the rectangles $R_b$ and $R_{b'}$ are ``sufficiently separated''. 
Our arguments depend on \eqref{absRelxsim3D} below and the fact that the arithmetic mean dominates the geometric mean.  
\vs

In Sections~\ref{m=domterms} and ~\ref{b3NEQb'33D}, we consider the cases when $b_3 \neq b'_3$ but one or both of the other $b_i$ are equal. Here, we say that the two distinct rectangles are within ``close proximity,'' and the arithmetic-geometric mean inequality does not suffice.  In 
Section~\ref{m=domterms}, a brute-force case analysis is conducted based on dominant terms when one $b_i = b'_i$.  
In Section ~\ref{b3NEQb'33D}, a more delicate analysis based on further decomposition of rectangles is required when both $b_i=b'_i$.
\vs

Note that when $b\neq b'$, there are a number of cases that are resolved through symmetry. For example, all instances of $b_1 = b'_1$ and $b_2 \neq b'_2$ are handled identically to those where $b_1 \neq b'_1$ and $b_2 = b'_2$ by relabeling the coordinates.
\vs

\subsection{Case $b'=b$ in 3 Dimensions: Dyadic Shells}\label{b=b'3D}
Inspecting the definition of $J(b,b')$ in \eqref{Jb}, we see that when $b'=b$, we have the single term
\begin{equation}\label{bb}
J(b,b') =J(b,b) =  \left(\frac{1}{\abs{E_q}}\right)     \int_{R_b}|x-y|^{-t} dy.
\end{equation} 
We partition the rectangle $R_b$ into its intersection with dyadic shells of the form $S_j=\{y \in \mathbb{R}^3 : 2^{-(j+1)} \leq \abs{x-y} \leq 2^{-j} \}$ for $j \in \mathbb{N}$. Note, we need only consider $j \geq \log_2q^{\tau}$, as otherwise the shell and the rectangle have empty intersection, and so
\begin{equation}\label{integral1} \int_{R_b}|x-y|^{-t} dy = \!\!\sum_{j\geq \log_2q^{\tau}} \int_{S_j\cap R_b} |x-y|^{-t} dy.
\end{equation}
As $j$ ranges, depending on whether $S_j$ is completely contained in the rectangle $R_b$ or not, we have two regions over the sum in $j$ to consider: (A) $0 \leq r = 2^{-j} \leq q^{-a-\tau}$ or (B) $q^{-a-\tau} \leq r = 2^{-j} \leq q^{-\tau}$. (See Figure~\ref{Sj&Rx})

\begin{figure}[ht]
  \centering
  \includegraphics[width=\linewidth]{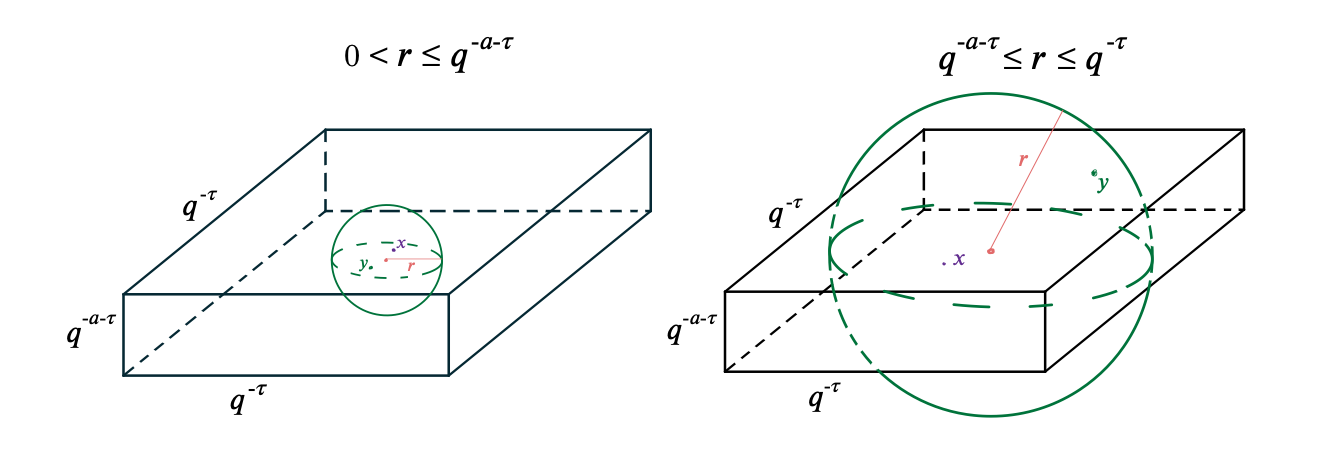}
\caption{Example of $S_j$ contained in $R_b$ and exceeding $R_b$}\label{Sj&Rx}
\end{figure}

In the first case, $j > \log_2 q^{a+\tau}$ and the shell $S_j \subseteq R_b$. In the latter case, $\log_2 q^\tau \leq j \leq \log_2q^{a+\tau}$, and $S_j \cap R_b \neq \emptyset$, but the shell is no longer completely contained in $R_b$. Accordingly, we further divide the integral in \eqref{integral1} as follows:
$$
\left( \sum_{j=\log_2q^{a+\tau}}^\infty  \int_{S_j\cap R_b}  \! \! \abs{x-y}^{-t}  dy \right) 
+ 
\left( \sum_{j=\log_2q^{\tau}}^{\log_2q^{a+\tau}}  \int_{S_j\cap R_b}  \!  \! \abs{x-y}^{-t} \!dy \right)  = I_A(b) + I_B(b).
$$
We will consider the sums in $I_A(b)$ and $I_B(b)$ separately.
\vs

\noindent
\textit{ Bounding $I_A$:}
When $j > \log_2 q^{a+\tau}$, the shell is completely contained in the rectangle so that $S_j\cap R_b = S_j$. 
Since $|x-y|^{-t} \sim 2^{jt}$ on $S_j$ and $|S_j| \sim 2^{-3j}$, we have
\begin{IEEEeqnarray}{cl}
I_A(b)  \,\, &\sim \! \sum_{j=\log_2q^{a+\tau}}^\infty 2^{jt} |S_j| \nonumber\\
&\sim \! \sum_{j=\log_2q^{a+\tau}}^\infty 2^{jt} 2^{-3j}. \label{sumInts3D}
\end{IEEEeqnarray}
Since $t<3$, \eqref{sumInts3D} is equivalent to
$$
\sum_{j=\log_2q^{a+\tau}}^\infty 2^{-j(3-t)}
\sim  (q^{a+\tau})^{-(3-t)}.
$$
%
%
\noindent
\textit{ Bounding $I_B$:}
When $\log_2 q^\tau \leq j \leq \log_2q^{a+\tau}$, the shell containing $S_j$ exceeds $R_b$ in the vertical direction. For such $j$, the intersection $S_j \cap R_b$ is contained in a cylinder of height $q^{-a-\tau}$ and radius $2^{-j}$, so that $|S_j \cap R_b| \lesssim 2^{-2j}q^{-a-\tau}$.
Thus, since $|x-y|^{-t} \sim 2^{jt}$ on $S_j$, 
\begin{IEEEeqnarray}{cl}
I_B(b) \,\, 
&\sim \sum_{j=\log_2q^{\tau}}^{\log_2q^{a+\tau}} 2^{jt}|S_j \cap R_x| \nonumber\\
&\lesssim \sum_{j=\log_2q^{\tau}}^{\log_2q^{a+\tau}} 2^{jt}2^{-2j}q^{-a-\tau}. \label{sumInts3D2}
\end{IEEEeqnarray}

Recalling that $t > 2$, \eqref{sumInts3D2} equals
\begin{IEEEeqnarray}{cl}
\sum_{j =\log_2q^{\tau}}^{\log_2q^{a+\tau}} 2^{j(t-2)}q^{-(a+\tau)}
\sim (q^{a+\tau})^{t-2} q^{-(a+\tau)}. \label{HNB1B3}
\end{IEEEeqnarray}
%

\noindent
\textit{ Bounding $I_A+I_B$:}
It now follows that 
\begin{IEEEeqnarray}{cl}
I_A(b) + I_B(b) \lesssim (q^{a+\tau})^{t-3}. \label{IA+IBbound3D}
\end{IEEEeqnarray}
\vs

Plugging \eqref{IA+IBbound3D} into $J(b,b)$ in \eqref{bb} and recalling that $|E_q| = N_R \cdot V_R$, $V_R = q^{-a-3\tau}$, $N_R =q^3$, and $a=\frac{3}{4}$, we see that $J(b,b)$ is bounded by a constant multiple of $1$ provided that $3t+4\tau t \leq 18$, which is guaranteed since $2.5\le t<3$ from \eqref{tbound} and $\tau \leq \frac{2a}{t-1}$ from \eqref{taubound3D}.
Thus, we have verified \eqref{reductionkt} when $b=b'$. 
\vs

\subsection{Sufficiently Separated Rectangles in 3-dimensions} \label{b3=b'3ALLneq3D}
In this section we address the cases where (1) $b'_3 = b_3$ and $b'_i \neq b_i$ for exactly one $i \in\{1, 2\}$ (see Section~\ref{m=13D}), (2) $b'_3 = b_3$ and $b'_i \neq b_i$ for both $i\in \{1,2\}$ (see Section~\ref{m=23D}), and (3) the case where $b'_i \neq b_i$ for all $i \in \{1,2,3\}$ (see Section~\ref{m=33D}).
\vs

Throughout the remainder of this proof, for $x\in R_b$ and $y\in R_{b'}$, we use the observation that 
\begin{IEEEeqnarray}{cl}
\abs{x-y} \, &\sim \abs{x_1-y_1} + \abs{x_2 - y_2} + \abs{x_3-y_3} \label{absRel3D} \\
& \sim \abs{\frac{b_1}{q^a} - \frac{b'_1}{q^a}} + \abs{\frac{b_2}{q^a} - \frac{b'_2}{q^a}} + \abs{\frac{b_3}{q^{2a}} - \frac{b'_3}{q^{2a}}}. \label{absRelxsim3D}
\end{IEEEeqnarray}
\vs

\subsubsection{Assume $b'_3=b_3$, but $b'_i \neq b_i$ for $i=1$ or 2} \label{m=13D}
Without loss of generality, assume $b'_1 \neq b_1$, but $b'_2=b_2$. Here,
$$J(b,b') =  \left(\frac{1}{\abs{E_q}}\right)    \sum_{\substack{b'_1=1 \\ b'_1 \neq b_1}}^{q^a}  \int_{R_{b'}}|x-y|^{-t} dy.$$
\par
We estimate $\abs{x-y}^{-t}$ by $\left(\frac{\abs{b_1-b'_1}}{q^a}\right)^{\!-t}$ so that 
\begin{align*}
J(b,b') &\sim  \frac{q^{at}}{|E_q|}  \sum_{\substack{b'_1=1 \\ b'_1 \neq b_1}}^{q^a}   \int_{R_{b'}} \abs{b_1-b'_1}^{-t}  dy.
\end{align*}

Observe that the integral in $y$ simply yields a factor of $V_R$.  Applying the integral test and recalling that $2.5\le t<3$, this expression is bounded by a constant multiple of 
\begin{equation*}\label{jbb'Bound}
 \frac{q^{at}V_R}{|E_q|} =  \frac{q^{at}}{N_R} = q^{at-3},
\end{equation*}
which is bounded by 1 since $a=\frac{3}{4}<1$ and $t<3$. 
\vs

\subsubsection{Assume $b'_1 \neq b_1$,  $b'_2 \neq b_2$, and $b'_3 = b_3$}\label{m=23D}

Here,
$$J(b,b') =  \left(\frac{1}{\abs{E_q}}\right) 
 \sum_{\substack{b'_1=1 \\ b'_1 \neq b_1}}^{q^a}  \sum_{\substack{b'_2=1 \\ b'_2 \neq b_2}}^{q^a}
  \int_{R_{b'}}|x-y|^{-t} dy.$$

Similar to the previous case, we estimate $\abs{x-y}$ using \eqref{absRelxsim3D}. However, unlike the previous case, we now must sum in $b'_2$ and a tighter estimate will be needed to assure the sums converge. 
Since, by assumption, $b'_3 =b_3$, we estimate 
\begin{equation*}
\abs{x-y} \geq \abs{x_1-y_1} + \abs{x_2 - y_2} \sim \abs{\frac{b_1}{q^a} - \frac{b'_1}{q^a}} + \abs{\frac{b_2}{q^a} - \frac{b'_2}{q^a}}.
\end{equation*}
\vs

Thus, using the fact that the arithmetic mean dominates the geometric mean, we bound $\abs{x-y}^{-t}$ by
\begin{IEEEeqnarray*}{cl}
2^{-t}\abs{x_1-y_1}^{-\frac{t}{2}}\abs{x_2 - y_2}^{-\frac{t}{2}} \sim \left(\frac{\abs{b_1-b'_1}}{q^a}\right)^{\!-\frac{t}{2}}\left(\frac{\abs{b_2-b'_2}}{q^a}\right)^{\!-\frac{t}{2}}.
\end{IEEEeqnarray*}
Hence, the expression $J(b,b')$ is bounded by
\begin{align*}
J(b,b') &\lesssim  \left(\frac{q^{at}\,V_R}{\abs{E_q}}\right) 
 \sum_{\substack{b'_1=1 \\ b'_1 \neq b_1}}^{q^a}  \sum_{\substack{b'_2=1 \\ b'_2 \neq b_2}}^{q^a}
\abs{b_1-b'_1}^{-\frac{t}{2}}  \abs{b_2-b'_2}^{-\frac{t}{2}}\\
&=   \left(\frac{q^{at}}{N_R }\right) 
 \Bigg[ \sum_{\substack{b'_1=1 \\ b'_1 \neq b_1}}^{q^a}  
 \abs{b_1-b'_1}^{-\frac{t}{2}} \! \Bigg]^2.
\end{align*}

Since $t > 2$, an application of the integral test yields the same bound as above,
$$\left(\frac{q^{at}}{N_R }\right) = q^{at-3},$$
which is bounded by 1 since $2 < t <3$ and $a = \frac{3}{4}$.

\begin{note}
Note that the seemingly similar case when $b'_3 \neq b_3$ and $b'_i \neq b_i$ for precisely one $i \in \{1,2\}$ \textit{cannot} be handled using an argument similar to the one above in Section \ref{m=23D}.  An alternative argument using dominant terms will be given to handle this case in Section \ref{m=domterms}.  
\end{note}


\subsubsection{Assume $b'_3 \neq b_3$ and both of the other $b'_i \neq b_i$}\label{m=33D} 
Here,
$$J(b,b') =  \left(\frac{1}{\abs{E_q}}\right) 
 \sum_{\substack{b'_1=1 \\ b'_1 \neq b_1}}^{q^a}  \sum_{\substack{b'_2=1 \\ b'_2 \neq b_2}}^{q^a}
 \sum_{\substack{b'_3=1 \\ b'_3 \neq b_3}}^{q^{2a}}
  \int_{R_{b'}}|x-y|^{-t} dy.$$
  
This case is handled analogously to the previous, except that we have a lengthier sum.
We will again make use of the fact that the arithmetic mean dominates the geometric mean:
\begin{IEEEeqnarray*}{cl}
\abs{x-y} \sim \abs{x_1-y_1} \!+\! \abs{x_2 - y_2}\! +\! \abs{x_3-y_3} \geq 3\!\abs{x_1-y_1}^{\frac{1}{3}}\abs{x_2 - y_2}^{\frac{1}{3}} \abs{x_3-y_3}^{\frac{1}{3}} \label{AMGM}
\end{IEEEeqnarray*}
where $x_i \sim \frac{b_i}{q^a}$ for $i=1,2$, $x_3 \sim \frac{b_3}{q^{2a}}$ and $y_i \sim \frac{b'_i}{q^a}$ for $i=1,2$, $y_3 \sim \frac{b'_3}{q^{2a}}$.
It follows that 
\begin{IEEEeqnarray*}{cl}
\abs{x-y}^{-t} \lesssim \left(\frac{\abs{b_1-b'_1}}{q^a}\right)^{\!\!-\frac{t}{3}}\left(\frac{\abs{b_2-b'_2}}{q^a}\right)^{\!\!-\frac{t}{3}} \left(\frac{\abs{b_3-b'_3}}{q^{2a}}\right)^{\!\!-\frac{t}{3}},
\end{IEEEeqnarray*}
and we can bound the integral $J(b,b')$ by 
\begin{align*}
J(b,b') &\lesssim  \Bigg(\frac{  \big(q^{a}\big)^{\frac{4}{3}t} \, V_R}{\abs{E_q}}\Bigg) \!\!
 \sum_{\substack{b'_1=1 \\ b'_1 \neq b_1}}^{q^a}  \sum_{\substack{b'_2=1 \\ b'_2 \neq b_2}}^{q^a}
 \sum_{\substack{b'_3=1 \\ b'_3 \neq b_3}}^{q^{2a}}
\abs{b_1-b'_1}^{-\frac{t}{3}}
\abs{b_2-b'_2}^{-\frac{t}{3}} 
\abs{b_3-b'_3}^{-\frac{t}{3}}\\
&=  \Bigg(\frac{\big(q^{a}\big)^{\frac{4}{3}t}}{N_R}\Bigg) 
\Bigg[\sum_{\substack{b'_1=1 \\ b'_1 \neq b_1}}^{q^a} 
 \abs{b_1-b'_1}^{\!-\frac{t}{3}}\Bigg]^2
\Bigg[  \sum_{\substack{b'_3=1 \\ b'_3 \neq b_3}}^{q^{2a}}
\abs{b_3-b'_3}^{\!-\frac{t}{3}}\Bigg].
\end{align*}

Since $t<3$, using the integral test, this is bounded by 
$$J(b,b') \lesssim
\left(\frac{\big(q^{a}\big)^{\frac{4}{3}t}}{N_R}\right) 
  \left( q^{a(1-\frac{t}{3})} \right)^2  
  \left( q^{2a(1-\frac{t}{3})} \right).$$

Plugging in $N_R= q^3$ and $a=\frac{3}{4}$, this expression is identically equal to $1$.

\subsection{``Numerous Rectangles" in 3-dimensions}\label{m=domterms}
In this section, we handle the case when 
$b'_3 \neq b_3$ and exactly one of the other $b'_i \neq b_i$
Upon inspection, it becomes clear that this case cannot be handled using the fact that the geometric mean dominates the arithmetic mean. Instead, an argument using dominant terms is given.  
\vs

We give the proof for $b'_1 \neq b_1$ and $b'_2 = b_2$, as both cases are symmetric.
Here,
$$J(b,b') =  \left(\frac{1}{\abs{E_q}}\right) 
 \sum_{\substack{b'_1=1 \\ b'_1 \neq b_1}}^{q^a}  \sum_{\substack{b'_3=1 \\ b'_3 \neq b_3}}^{q^{2a}}
  \int_{R_{b'}}|x-y|^{-t} dy.$$

For $x\in R_b$ and $y\in R_{b'}$, we estimate
 $x_1 \sim \frac{b_1}{q^a}$, $x_2 \sim \frac{b_2}{q^a}$, and $x_3 \sim \frac{b_3}{q^{2a}}$, as well as
 $y_1 \sim \frac{b'_1}{q^a}$,  $y_2 \sim \frac{b_2}{q^a}$, and $y_3 \sim \frac{b'_3}{q^{2a}}$, and we appeal to 
 \eqref{absRel3D}.  
 \vs
 
For simplicity, we will conduct the estimate for $b=(0,0,0)$, as the more general argument follows similarly.   Thus, 
for $y\in R_{b'}$, $\abs{x-y}\sim  \frac{b'_1}{q^a}+\frac{b'_3}{q^{2a}}$, where $\frac{b'_1}{q^a} \in \big\{\frac{1}{q^a}, \frac{2}{q^a}, \dots, \frac{q^a}{q^a} \big\}$ and $\frac{b'_3}{q^{2a}} \in \big\{\frac{1}{q^{2a}}, \frac{2}{q^{2a}}, \dots, \frac{q^{2a}}{q^{2a}} \big\}$.
This naturally leads to two cases: 
\begin{description}
\item[Case I] The horizontal distance in $z_1$ dominates the vertical: $\frac{b'_1}{q^a} \geq \frac{b'_3}{q^{2a}}$.\vs

\item[Case II] The vertical distance between $x$ and $y$ dominates: $\frac{b'_3}{q^{2a}}>\frac{b'_1}{q^a}$.\vs
\end{description}

We now proceed to bound $J(b,b')$ in each case.

\begin{itemize}

\item[Case I:] Assume $\frac{b'_1}{q^a} \geq \frac{b'_3}{q^{2a}}$, and estimate $\abs{x-y}$ by $\frac{b'_1}{q^a}$. Observe that since $\frac{b'_1}{q^a} \geq \frac{b'_3}{q^{2a}} \Leftrightarrow b'_3 \leq b'_1q^a$, we need only sum the $b'_3$ to $b'_1q^a \leq q^{2a}$, resulting in the following estimation: 
\begin{IEEEeqnarray*}{cl}
J(b,b') &=  \left(\frac{1}{\abs{E_q}}\right)   \sum_{b'_1=1}^{q^a} \sum_{b'_3=1}^{b'_1q^a} \int_{R_{b'}} \abs{x-y}^{-t}dy \nonumber\\
&\sim \left(\frac{V_R}{\abs{E_q}}\right) \sum_{b'_1=1}^{q^a} \sum_{b'_3=1}^{b'_1q^a} \left(\frac{b'_1}{q^a}\right)^{\!\!-t} \nonumber\\
&= \frac{1}{N_R} \,q^{at} \sum_{b'_1=1}^{q^a} (b'_1q^a)(b'_1)^{-t}. \label{2(B)I 3D}
\end{IEEEeqnarray*}
Applying the integral test and recalling that $2.5\le <3$, this expression is bounded by 
$$ \frac{1}{N_R} \,q^{at} q^a \lesssim 1,$$
since $N_R= q^3$ and $a=\frac{3}{4}$
\vs

\item[Case II:]  Assume $\frac{b'_1}{q^a} < \frac{b'_3}{q^{2a}}$, so we estimate $\abs{x-y}$ by $\frac{b'_3}{q^{2a}}$. 
Observe that since $\frac{b'_1}{q^a}~<~\frac{b'_3}{q^{2a}} \Leftrightarrow b'_3 > b'_1q^a$, we need only sum $b'_3$ from $b'_1q^a$ to the upper bound, $q^{2a}$. Now

\begin{IEEEeqnarray*}{cl}
J(b,b') &=  \left(\frac{1}{\abs{E_q}}\right)   \sum_{b'_1=1}^{q^a} \sum_{b'_3=b'_1q^a}^{q^{2a}} \int_{R_{b'}} \abs{x-y}^{-t}dy \nonumber\\
 &\sim  \left(\frac{V_R}{\abs{E_q}}\right)  \sum_{b'_1=1}^{q^a} \sum_{b'_3=b'_1q^a}^{q^{2a}}  \left(\frac{b'_3}{q^{2a}}\right)^{\!\!-t} \nonumber\\
&= \left(\frac{1}{N_R}\right) \,q^{2at} \sum_{b'_1=1}^{q^a} \sum_{b'_3=b'_1q^a}^{q^{2a}}(b'_3)^{-t}. \label{2(B)II 3D}
\end{IEEEeqnarray*}

Applying the integral test for $2<t<3$, 
$$\sum_{b'_3=b'_1q^a}^{q^{2a}} (b'_3)^{-t} \sim \big(b'_1q^a\big)^{1-t}.$$

Plugging this back into the bound for $J(b,b')$ we have
$$J(b,b') \lesssim \left(\frac{1}{N_R}\right) \,q^{2at} \sum_{b'_1=1}^{q^a} \big(b'_1q^a\big)^{1-t}.$$

Again applying the integral test to the sum in $b'_1$, we see that 
$$J(b,b') \lesssim \left(\frac{1}{N_R}\right) \,q^{2at} \big(q^a\big)^{1-t} = \frac{1}{N_R} \,q^{at} q^a,$$
the same upper bound as in Case I. 

\end{itemize}


\subsection{``High Proximity Rectangles" in 3-dimensions} \label{b3NEQb'33D}
Finally, we consider the most complex case when 
$b'_3 \neq b_3$ and both of the other $b'_i = b_i$ for $i=1$ and 2. 
Upon inspection, it becomes clear that this case cannot be handled using the techniques of the previous sections. Instead, a finer decomposition of the rectangle $R$ into cubes is required.  
\vs

Observe that the minimal vertical distance between distinct rectangles $R_b$ and $R_{b'}$ is significantly smaller than the minimal horizontal distance.  
The minimal distance between distinct rectangles is attained in this case.  
Consequently, this case requires a more cautious consideration of the locations of $x$ and $y$ within $R_b$ and $R_{b'}$, respectively. We will subdivide each rectangle into $q^{3 + 2a}$ smaller cubes of side-length $\frac{1}{q^{a+\tau}}$. 
Let 
$C_{\ell,b}$ be the cube with side-length $\frac{1}{q^{a+\tau}}=q^{-a-\tau}$, centered at $\big(\frac{b_1}{q^a}+\frac{\ell_1 }{q^{a+\tau}}, \frac{b_2}{q^a}+\frac{\ell_2 }{q^{a+\tau}}, \frac{b_3}{q^{2a}}\big)$, where $\ell_1, \ell_2 \in \{0, 1, \dots, q^a\}$. Let $V_C$ denote the volume of such a cube.
\vs

Hence,
$$J(b,b') =  \left(\frac{1}{\abs{E_q}}\right) 
\sum_{\substack{b'_3=1 \\ b'_3 \neq b_3}}^{q^{2a}}
  \int_{R_{b'}}|x-y|^{-t} dy.$$

For $x\in R_b$ and $y\in C_{\ell,b'}$, 
we again use \eqref{absRel3D}, but we make the approximation of $x_i \sim \frac{b_i}{q^a}$ for $i=1,2$, $x_3 \sim \frac{b_3}{q^{2a}}$ and $y_i \sim \frac{b'_i}{q^a}+\frac{\ell_i}{q^{a+\tau}}$ for $\ell_i \in \{0,1, \dots, q^a\}$, for $i=1,2$, $y_3 \sim \frac{b'_3}{q^{2a}}$. Since, by assumption, $b'_1=b_1$ and $b'_2=b_2$, we thus have:
\begin{IEEEeqnarray*}{cl}
|x-y| \sim \frac{\ell_1}{q^{a+\tau}} + \frac{\ell_2}{q^{a+\tau}} + \abs{\frac{b_3}{q^{2a}} - \frac{b'_3}{q^{2a}}}.
\end{IEEEeqnarray*}

As in the previous case, we will conduct the estimate for $b=(0,0,0)$, as the more general argument follows from a similar argument.   Now
$$
|x-y| \sim \frac{\ell_1}{q^{a+\tau}} + \frac{\ell_2}{q^{a+\tau}} + \abs{\frac{b'_3}{q^{2a}}}.$$

This leads to four cases:
\begin{description}
\item[Case I] The horizontal distances in both directions
dominate the vertical: $\frac{\ell_i }{q^{a+\tau}} \geq \frac{b'_3}{q^{2a}}$ for both $i = 1,2$, which occurs when $\ell_i \geq b'_3 q^{\tau-a}$.\\
Note that, since $\ell_i \leq q^a$, this is not possible if $b'_3>q^{2a-\tau}$. 
\vs

\item[Case II] The horizontal distance in $z_1$ dominates the vertical: $\frac{\ell_1 }{q^{a+\tau}} \geq \frac{b'_3}{q^{2a}}$, but $\frac{\ell_2 }{q^{a+\tau}} < \frac{b'_3}{q^{2a}}$. \\

\item[Case III] The horizontal distance in $z_2$ dominates the vertical: $\frac{\ell_2 }{q^{a+\tau}} \geq \frac{b'_3}{q^{2a}}$, but $\frac{\ell_1 }{q^{a+\tau}} < \frac{b'_3}{q^{2a}}$, which is equivalent to Case II by symmetry. \\

\item[Case IV] The vertical distance between $x$ and $y$ dominates: $\frac{\ell_i }{q^{a+\tau}} < \frac{b'_3}{q^{2a}}$ for both $i = 1,2$, which occurs when 
$\ell_i<\frac{b'_3q^{\tau}}{q^a}$.

\end{description}

\begin{itemize}
\item[Case I:] Assume $\frac{\ell_1 }{q^{a+\tau}},\frac{\ell_2 }{q^{a+\tau}} \geq \frac{b'_3}{q^{2a}}$, so that 
\begin{equation}\label{abs_aprx}
\abs{x-y} \sim \frac{1}{q^{a+\tau}}\abs{\ell_1+\ell_2}.
\end{equation}
 Recall that $\ell_i \leq q^a$, so for the horizontal distance to be dominant, we must have $b'_3 \leq \ell_iq^{a-\tau} \leq q^{2a-\tau}$. Thus, we need only sum $b'_3$ up to $q^{2a-\tau}$ since the vertical distance is dominant for larger $b'_3$, and we have

$$J(b,b')\cdot \abs{E_q} = 
 \sum_{b'_3=1}^{q^{2a-\tau}} 
 \sum_{\ell_1=\frac{b'_3q^{\tau}}{q^a}}^{q^a}   \sum_{\ell_2=\frac{b'_3q^{\tau}}{q^a}}^{q^a}
 \int_{C_{\ell,b'}} \abs{x-y}^{-t}dy.$$

We further divide this sum into two sums, one where $\ell_1 \geq \ell_2$, so $\abs{x-y} \sim \frac{1}{q^{a+\tau}}\abs{\ell_1}$ by \eqref{abs_aprx}, and the other where $\ell_1 \leq \ell_2$, so $\abs{x-y} \sim \frac{1}{q^{a+\tau}}\abs{\ell_2}$.  Summing in $\ell_2$, we have
\begin{IEEEeqnarray}{cl}
 \sum_{b'_3=1}^{q^{2a-\tau}} & \sum_{\ell_1=\frac{b'_3q^{\tau}}{q^a}}^{q^a}\sum_{\ell_2=\frac{b'_3q^{\tau}}{q^a}}^{\ell_1} \int_{C_{\ell,b'}} \abs{x-y}^{-t}dy 
+ 
\sum_{b'_3=1}^{q^{2a-\tau}} \sum_{\ell_1=\frac{b'_3q^{\tau}}{q^a}}^{q^a}\sum_{\ell_2=\ell_1}^{q^a} \int_{C_{\ell,b'}} \abs{x-y}^{-t}dy  \nonumber\\
&\sim \sum_{b'_3=1}^{q^{2a-\tau}} \sum_{\ell_1=\frac{b'_3q^{\tau}}{q^a}}^{q^a}\sum_{\ell_2=\frac{b'_3q^{\tau}}{q^a}}^{\ell_1} \!\!\!V_C\left(\frac{1}{q^{a+\tau}}\abs{\ell_1}\right)^{\!\!\!-t} \!+ 
\sum_{b'_3=1}^{q^{2a-\tau}} \sum_{\ell_1=\frac{b'_3q^{\tau}}{q^a}}^{q^a}\sum_{\ell_2=\ell_1}^{q^a} \!V_C\left(\frac{1}{q^{a+\tau}}\abs{\ell_2}\right)^{\!\!\!-t} \label{ell,t<1 3D}\nonumber\\
&\sim V_C \!\left(\!\frac{1}{q^{a+\tau}}\!\right)^{\!\!\!-t} \!\,\,\sum_{b'_3=1}^{q^{2a-\tau}} \sum_{\ell_1=\frac{b'_3q^{\tau}}{q^a}}^{q^a} \!\!\abs{\ell_1}^{-t}\!\left(\!\ell_1-\frac{b'_3q^{\tau}}{q^a}\!\right) \!+ 
V_C \bigg(\!\frac{1}{q^{a+\tau}}\!\bigg)^{\!\!\!-t} \,\,\!\sum_{b'_3=1}^{q^{2a-\tau}} \sum_{\ell_1=\frac{b'_3q^{\tau}}{q^a}}^{q^a} \!\!\ell_1^{1-t} \nonumber\\ \label{ellDom13D}
&= I + II \nonumber
\end{IEEEeqnarray}
since $t>1$.
We next apply the integral test to sum in $\ell_1$, recognizing 
$I \leq II$ and $t> 2$, we have
\begin{IEEEeqnarray*}{cl}
\int_{b'_3/q^{a-\tau}}^{q^a} x^{-t}\left(x-\frac{b'_3q^{\tau}}{q^a}\right)dx \ \leq \int_{b'_3/q^{a-\tau}}^{q^a} x^{1-t}dx \sim \left(\frac{b'_3q^{\tau}}{q^a}\right)^{\!\!2-t} \!.
\end{IEEEeqnarray*}

Thus, 
\begin{IEEEeqnarray*}{cl}
I + II \, &\lesssim V_C \left(\frac{1}{q^{a+\tau}}\right)^{\!\!\!-t} \,\sum_{b'_3=1}^{q^{2a-\tau}} \left(\frac{b'_3q^{\tau}}{q^a}\right)^{\!\!2-t} \\
&= V_C \, q^{2at}\left(\frac{q^\tau}{q^a}\right)^{\!\!2} \,\sum_{b'_3=1}^{q^{2a-\tau}} (b'_3)^{2-t}.
\end{IEEEeqnarray*}

Once again utilizing the integral test and recognizing that $t<3$,
we bound $J(b,b') \cdot \abs{E_q}$ by 
\begin{equation}\label{later}
V_C \, q^{2at}q^{2\tau-2a}(q^{2a-\tau})^{3-t}=V_C \, q^{2\tau-2a}q^{6a}(q^{-\tau})^{3-t} = V_C \,q^{4a}(q^\tau)^{t-1}.\end{equation}

Seeking a bound for $J(b,b')$, and recalling that $\abs{E_q}=N_C \cdot V_C$, we have 
\begin{IEEEeqnarray*}{cl}
J(b,b') \lesssim 
\frac{1}{|E_q|}V_C \, q^{4a} \big(q^{\tau}\big)^{t-1} = \frac{V_C \,q^{4a} \big(q^{\tau}\big)^{t-1}}{N_C \cdot V_C} =
\frac{(q^\tau)^{t-1}}{q^{2a}},
\end{IEEEeqnarray*}
since $N_C=q^{6a}$. 
This is bounded since $\tau \leq \frac{2a}{t-1}$ by \eqref{taubound3D}.
 \vs

\item[Case II:] Assume $\frac{\ell_1 }{q^{a+\tau}} \geq \frac{b'_3}{q^{2a}}$, and $\frac{\ell_2 }{q^{a+\tau}} < \frac{b'_3}{q^{2a}}$, we omit Case III, as it is equivalent to Case II by symmetry. For Case II, we have $\abs{x-y} \sim \frac{1}{q^{a+\tau}}\abs{\ell_1}$.
\begin{IEEEeqnarray}{cl}
J(b,b') \cdot \abs{E_q}&=  \sum_{b'_3=1}^{q^{2a-\tau}} \sum_{\ell_1=\frac{b'_3q^{\tau}}{q^a}}^{q^a} \sum_{\ell_2=1}^{\frac{b'_3q^{\tau}}{q^a}} \int_{C_{\ell,b'}} \abs{x-y}^{-t} dy \nonumber\\
&\sim \sum_{b'_3=1}^{q^{2a-\tau}} \sum_{\ell_1=\frac{b'_3q^{\tau}}{q^a}}^{q^a}\sum_{\ell_2=1}^{\frac{b'_3q^{\tau}}{q^a}} V_C \left(\frac{1}{q^{a+\tau}}\abs{\ell_1}\right)^{\!\!-t} \nonumber\\
&= V_C \left(\frac{1}{q^{a+\tau}}\right)^{\!\!-t} \! \left(\frac{1}{q^{a-\tau}}\right)\sum_{b'_3=1}^{q^{2a-\tau}}b'_3 \sum_{\ell_1=\frac{b'_3q^{\tau}}{q^a}}^{q^a} \abs{\ell_1}^{-t}. \label{ell1DomBound13D}
\end{IEEEeqnarray}

Applying the integral test to sum in $\ell_1$, and since $t> 2$, we see that
\begin{IEEEeqnarray*}{cl}
\int_{\frac{b'_3q^{\tau}}{q^a}}^{q^{a}} x^{-t}dx = x^{1-t} \biggr|_{\frac{b'_3q^{\tau}}{q^a}}^{q^{a}} \sim \left(\frac{b'_3q^{\tau}}{q^a}\right)^{\!\!1-t}.
\end{IEEEeqnarray*}
Thus, (\ref{ell1DomBound13D}) is approximately
\begin{IEEEeqnarray*}{cl}
V_C \left(\frac{1}{q^{a+\tau}}\right)^{\!\!-t} \left(\frac{q^\tau}{q^a}\right)\sum_{b'_3=1}^{q^{2a-\tau}}b'_3 \left(\frac{b'_3q^{\tau}}{q^a}\right)^{\!\!1-t} 
&= V_C \,q^{2at} q^{2\tau-2a} \sum_{b'_3=1}^{q^{2a-\tau}} (b'_3)^{2-t}.
\end{IEEEeqnarray*}

Once again, we apply the integral test to the sum in $b'_3$:
\begin{IEEEeqnarray*}{cl}
\int_{1}^{q^{2a-\tau}} x^{2-t}dx = x^{3-t} \biggr|_{1}^{q^{2a-\tau}} \sim \big(q^{2a-\tau}\big)^{3-t}
\end{IEEEeqnarray*}
since $t<3$.
Thus we have
\begin{IEEEeqnarray*}{cl}
J(b,b') \cdot \abs{E_q} \, &\sim V_C \,q^{2at} q^{2\tau-2a} (q^{2a-\tau})^{3-t} \\
&= V_C \, q^{4a} (q^{-\tau})^{1-t}.
\end{IEEEeqnarray*}
Noticing that this is the same bound as that obtained in \eqref{later} in Case 1, we conclude that $J(b,b') \lesssim 1$
since $\tau \leq \frac{2a}{t-1}$ by \eqref{taubound3D}.
 \vs

\item[Case IV:] Assume $\frac{\ell_1 }{q^{a+\tau}},\frac{\ell_2 }{q^{a+\tau}} < \frac{b'_3}{q^{2a}}$. We will estimate $\abs{x-y}$ by $\frac{b'_3}{q^{2a}}$.
Observe that for $i=1,2$,
\begin{IEEEeqnarray}{cl}
\ell_i \leq \min\left\{\frac{b'_3q^{\tau}}{q^a},q^a\right\} =
\left\{ \begin{IEEEeqnarraybox}[][c]{l?s} 
\IEEEstrut
q^a & for $b'_3>q^{2a-\tau}$\\
\frac{b'_3q^{\tau}}{q^a} & for $b'_3 \leq q^{2a-\tau}$.
\nonumber
\IEEEstrut
\end{IEEEeqnarraybox} 
\right.
\end{IEEEeqnarray}
We utilize this to subdivide our sum in $b'_3$ as follows:
\begin{enumerate}
\item $b'_3>q^{2a-\tau}$, for which we must sum $\ell_i \in \{1, \dots, q^a\}$, since $\ell_i \leq q^a < \frac{b'_3q^{\tau}}{q^a}$, and 
\item $b'_3 \leq q^{2a-\tau}$, for which we need only sum $\ell_i \in \big\{1, \dots, \frac{b'_3q^{\tau}}{q^a}\big\}$, since $\ell_i \leq \frac{b'_3q^{\tau}}{q^a} \leq q^a$.\\
\end{enumerate}

Hence,
\begin{IEEEeqnarray}{cl}
J(b,b') \cdot \abs{E_q} \,&= \sum_{b'_3=1}^{q^{2a}} \sum_{\ell_1=1}^{\frac{b'_3q^{\tau}}{q^a}} \sum_{\ell_2=1}^{\frac{b'_3q^{\tau}}{q^a}} \int_{C_{\ell,b'}} \abs{x-y}^{-t} dy \nonumber\\
&= \sum_{b'_3=q^{2a-\tau}}^{q^{2a}} \sum_{\ell_1=1}^{q^a} \sum_{\ell_2=1}^{q^a} \int_{C_{\ell,b'}} \abs{x-y}^{-t} dy + \sum_{b'_3=1}^{q^{2a-\tau}} \sum_{\ell_1=1}^{\frac{b'_3q^{\tau}}{q^a}} \sum_{\ell_2=1}^{\frac{b'_3q^{\tau}}{q^a}}\int_{C_{\ell,b'}} \abs{x-y}^{-t} dy \nonumber\\
&= \sum_{b'_3=q^{2a-\tau}}^{q^{2a}} q^a q^a V_C \left(\frac{b'_3}{q^{2a}}\right)^{\!\!-t} + \sum_{b'_3=1}^{q^{2a-\tau}} \frac{b'_3q^{\tau}}{q^a} \frac{b'_3q^{\tau}}{q^a} V_C \left(\frac{b'_3}{q^{2a}}\right)^{\!\!-t} \nonumber\\
&= V_C \,q^{2a} q^{2at} \sum_{b'_3=q^{2a-\tau}}^{q^{2a}} (b'_3)^{-t} + V_C \, q^{2\tau-2a}q^{2at}\sum_{b'_3=1}^{q^{2a-\tau}} (b'_3)^{2-t}. \label{b3Dom1Bound3D}
\end{IEEEeqnarray}

First, note that since $t>2 $,
\begin{IEEEeqnarray*}{cl}
\sum_{b'_3=q^{2a-\tau}}^{q^{2a}} (b'_3)^{-t} = \int_{q^{2a-\tau}}^{q^{2a}} x^{-t}dx 
\sim (q^{2a-\tau})^{1-t}.
\end{IEEEeqnarray*}

Next, we see that 
\begin{IEEEeqnarray*}{cl}
\sum_{b'_3=1}^{q^{2a-\tau}} (b'_3)^{2-t} = \int_{1}^{q^{2a-\tau}} x^{2-t}dx 
\sim (q^{2a-\tau})^{3-t}.
\end{IEEEeqnarray*}
since $t<3$.
So, we have 
$$(\ref{b3Dom1Bound3D}) \sim V_C \, q^{4a} (q^\tau)^{t-1} + V_C \, (q^\tau)^{t-1}q^{4a}.$$
Recognizing this as the same bound as in the last two cases, we conclude that  $J(b,b')$ is bounded since $\tau \leq \frac{2a}{t-1} $ by \eqref{taubound3D}.
\end{itemize}

In conclusion, we have verified \eqref{reductionkt} when $b=b'$.

\vskip.125in 
\section{Sharpness of the Main result and An improvement for small Delta}\label{sharp_section}
In this section, we use the error estimate of Garg, Nevo, and the second listed author in \cite{GNT}
to show that our main result is sharp when  $\delta$ obeys a lower bound dependent on $\al$ and $n$. 
Further, for sufficiently small values of $\de$ where sharpness is not attained, we demonstrate a method to improve Theorem \ref{Heisenberg}.  
As an example, when $\al=2$, our main theorem is sharp when $\de> R^{-\frac{(n-2)}{n}}$, but an improvement is possible when $\de< R^{-\frac{(n-2)}{n}}$.

\subsection{Sharpness}
We begin by establishing sharpness of Theorem \ref{Heisenberg} in several regimes by showing that the inequality in \eqref{Heisenberg_count} is in fact an equality for a range of $\de$ dependent on $\al$ and $n$.  
\begin{proposition}\label{sharp}
Let $d$ and $\al $ be integers, and set $n=2d+1$. 
Then
\begin{equation}\label{sharp_eq}
 \#\big( \{ m \in \mathbb{Z}^{2d} \ \mathsf{x} \  \mathbb{Z} :  
R-\delta \leq \Vert m \Vert_\al \leq R+ \delta \}\big) 
\sim R^{n}\delta,
 \end{equation} 
 for $R>1$ and $n\geq 3$, 
when $\al\geq 2(n-1)$ and $\delta \geq R^{- \frac{1}{(\al-1)}}$; 
when  $2< \al \le 2(n-1)$ and $\delta\geq R^{-\frac{1}{(2n-3)}}$;
or 
when $\al=2$ and  $\delta \geq R^{-\frac{ (n-2) }{n} }$.  
\end{proposition}
\vs

 To prove \eqref{sharp_eq}, 
  recall the definition of the Heisenberg norm ball $B_R^\alpha= B_R^{\alpha, 1}$ in \eqref{Hball}, and recall the volume computation,
$ |B_R^\alpha| = R^{2d+2}|B_1^\alpha|.$
Defining 
\begin{equation}\label{error_eq}
E(R) := \#\big(\mathbb{Z}^{2d+1} \cap B_{R}^{\alpha}\big) - |B_R^\alpha|,
\end{equation}
we have 
\begin{align*}
 \#\big( \{ m \in \mathbb{Z}^{2d} \ \mathsf{x} \  \mathbb{Z} :  R \leq \Vert m \Vert_\al \leq R+ \delta \}\big) 
&= \#\big(\mathbb{Z}^{2d+1} \cap  B_{R+\delta}^{\alpha}\big) - \#\big(\mathbb{Z}^{2d+1} \cap  B_{R}^{\alpha}\big) \\
&= |B_{R+\delta}^\alpha| + E(R+\delta) - \big(|B_R^\alpha| +E(R)\big) \\
&= |B_{R+\delta}^\alpha| - |B_R^\alpha| + E(R+\delta) - E(R)\\
&\geq cR^{2d+1}\delta -  \left( |E(R+\delta)| +|E(R)| \right)
\end{align*}
for some constant $c>0$, where in the last line we observe that $ |B_{R+\delta}^\alpha| - |B_R^\alpha| \le cR^{2d+1}\delta$ and subtract the absolute values of the errors.  
We now verify that the first term, $cR^{2d+1}\de$, dominates the error terms and hence is the desired lower bound. 
\vs

The following bounds on the error were achieved by Garg, Nevo, and Taylor in \cite[Theorem 1.1]{GNT} for all $\alpha$: \begin{equation}
|E(R)| \lesssim \left\{ \begin{IEEEeqnarraybox}[][c]{l?s} 
\IEEEstrut
R^{2d}& for $d\geq 1$ ($n\geq 3$) and $\al=2$\\
R^{2 + \max{\{0, \de(\al) \}} } \log(R)& for $d=1$ ($n=3$) and $\al>2$ \\
R^4\big(\log(R)\big)^{2/3} & for $d=2$ ($n=5$) and $\al > 2$\\
R^{2d} & for $d\geq 3$ ($n \geq 7$) and $\al > 2$,
\nonumber
\IEEEstrut
\end{IEEEeqnarraybox}
\right.  
\end{equation}
where $\delta(\al) = \frac{2\al - 8}{3\al-4}$ so that in particular $\max{\{0, \de(\al) \}} = 0$ for $ \al \le 4$.  
\vs

With these bounds in tow, it is not hard to verify that $|E(R)| \le \frac{c}{100}R^{n}\delta$ when $\delta \geq C\max\{ \frac{\log{R}}{R},  R^{-1+ \max{\{0, \de(\al) \}} } \}$,  for a sufficiently small constant $C$. 
A computation shows that the lower bounds on $\de$ assumed in Proposition \ref{sharp} are more strict, and we conclude then that 
\begin{equation}\label{lower}
 \#\big( \{ m \in \mathbb{Z}^{2d} \ \mathsf{x} \  \mathbb{Z} :  R \leq \Vert m \Vert_\al \leq R+ \delta \}\big) 
\gtrsim cR^{2d+1}\delta.
\end{equation}
In summary, the lower bound in \eqref{lower} agrees with the upper bound in  \eqref{Heisenberg_count} 
for the range of $\de$ specified in Proposition \ref{sharp}, 
and Theorem \ref{Heisenberg} is sharp in this regime.

\subsection{An improvement to Theorem \ref{Heisenberg} for small $\delta$}\label{sec_alt}
Utilizing the results from \cite[Theorem 1.1]{GNT}, we can achieve the following bounds on the number of lattice points within a $\delta$-thickening of a given Heisenberg sphere of radius $R$.

\begin{proposition} \label{useGNT}
For 
$d\geq 1$, $R>1$, and $\de \in [0,1)$,
\begin{equation}
 \#\big( \{ m \in \mathbb{Z}^{2d} \ \mathsf{x} \  \mathbb{Z} :  R \leq \Vert m \Vert_\al \leq R+ \delta \}\big) 
\lesssim \max\{R^{n}\delta, \abs{E(R)}\}
\end{equation}
where $\abs{E(R)}$ is the error term defined above from \cite{GNT}.
\end{proposition}

\begin{proof}
Following \eqref{error_eq}, similar to the lower bound computation above, we have
\begin{align*}
 \#\big( \{ m \in \mathbb{Z}^{2d} \ \mathsf{x} \  \mathbb{Z} :  R \leq \Vert m \Vert_\al \leq R+ \delta \}\big) 
&= \#\big(\mathbb{Z}^{2d+1} \cap  B_{R+\delta}^{\alpha}\big) - \#\big(\mathbb{Z}^{2d+1} \cap  B_{R}^{\alpha}\big) \\
&= |B_{R+\delta}^\alpha| + E(R+\delta) - \big(|B_R^\alpha| +E(R)\big) \\
&= |B_{R+\delta}^\alpha| - |B_R^\alpha| + E(R+\delta) - E(R)\\
&\leq cR^{2d+1}\delta + \abs{E(R+\delta)} + \abs{E(R)}
\end{align*}
for some constant $c>0$, where in the last line we observe that $ |B_{R+\delta}^\alpha| - |B_R^\alpha| \le cR^{2d+1}\delta$ and add the absolute values of the errors.
\end{proof}

Observe that, for small $\delta>0$, Proposition~\ref{useGNT} improves on Theorem \ref{Heisenberg}. 
For instance, when $\al=2$ and $\delta \leq R^{-\frac{ (n-2) }{ n } }$, 
Theorem \ref{Heisenberg} yields an upper bound of $R^{n-\frac{ (n-2) }{ n } }$, 
while 
Proposition \ref{useGNT} yields an improved upper bound of
$\max\{ R^{n-1}, R^n \de\}$ for $\delta \leq R^{-\frac{ (n-2) }{ n } }$.  
We note that such improvements come at the cost of a slightly more complicated proof as the error bounds on which Proposition \ref{useGNT} rely depend on the use of intricate Bessel function estimates, unlike that of Theorem~\ref{Heisenberg}.   Nevertheless, the two results working together yield the current best known upper bounds for counting lattice points near the Heisenberg spheres, and in many regimes, as detailed above, Theorem \ref{Heisenberg} is sharp.  
\vs
%


\end{document}